\documentclass[a4paper, 12pt, onecolumn]{article} \textwidth 148mm
\usepackage{amsmath}
\usepackage{amsthm}
\usepackage{amsfonts}
\usepackage{bbm}
\usepackage{CJK}
\usepackage{fancyhdr}
\usepackage{graphicx}
\usepackage{geometry}
\usepackage{psfrag}
\usepackage{amsfonts,amsmath,amsthm, amssymb}
\usepackage{latexsym, euscript, epic, eepic}
\usepackage{time}
\usepackage{txfonts}
\usepackage{colortbl}
\usepackage{stmaryrd}
\usepackage{mathrsfs}
\usepackage{txfonts}
\usepackage{amsfonts}
\usepackage{color}
\usepackage{lineno}
\usepackage{amsmath}
\allowdisplaybreaks[3]
\usepackage{indentfirst,latexsym,bm}

 \numberwithin{equation}{section}
\begin{document}
\newtheorem{theorem}{Theorem}[section]
\newtheorem{proposition}[theorem]{Proposition}
\newtheorem{remark}[theorem]{Remark}
\newtheorem{corollary}[theorem]{Corollary}
\newtheorem{definition}{Definition}[section]
\newtheorem{lemma}[theorem]{Lemma}

\title{\large Some Quantitative Properties of Solutions to the Buckling Type Equation}

\author{\normalsize Long Tian\\
\scriptsize School of Mathematics and Statistics, Nanjing University of Science $\&$ Technology, Jiangsu, Nanjing 210094, China.\\
\scriptsize Email: tianlong19850812@163.com\\
\normalsize Xiaoping Yang\\
\scriptsize Department of Mathematics, Nanjing University, Jiangsu, Nanjing 210008, China.\\
\scriptsize Email: xpyang@nju.edu.cn\\
}
\date{}
\maketitle
 \fontsize{12}{22}\selectfont\small
 \paragraph{Abstract:} In this paper, we investigate the quantitative unique continuation, propagation of smallness and measure bounds of nodal sets of solutions to the Buckling type equation $\triangle^2u+\lambda\triangle u-k^2u=0$ in a bounded analytic domain $\Omega\subseteq\mathbb{R}^n$ with the homogeneous boundary conditions $u=0$ and $\frac{\partial u}{\partial\nu}=0$ on $\partial\Omega$, where $\lambda,\ k$ are nonnegative real constants, and  $\nu$ is the outer unit normal vector on $\partial\Omega$. We obtain that, the upper bounds for the maximal vanishing order of $u$ and the $n-1$ dimensional Hausdorff measure of the nodal set of $u$ are both $C(\sqrt{\lambda}+\sqrt{k}+1)$, where $C$ is a positive constant only depending on $n$ and $\Omega$. Moreover, we also give a quantitative result of the propagation of smallness of $u$.

\\[10pt]

\vspace{1cm}\fontsize{12}{22}\selectfont

%%%%%%%%%%%%%%%%%%%%%%%%%%%%%%%%%%%%%%%%%%%%%%%%%%%%%%%%%%%%%%%%%

\section{Introduction}

In this paper, we will consider the quantitative unique continuation property and upper bounds of the nodal sets of solutions to the Buckling type equation with homogeneous boundary conditions in some bounded analytic domain $\Omega\subseteq\mathbb{R}^n$. Here, a bounded domain $\Omega$ is said to be analytic if there exists a positive constant $\delta$ such that for any $x_0\in\partial\Omega$, $B_{\delta}(x_0)\cap\partial\Omega$ is an $(n-1)-$dimensional analytic hypersurface of $\mathbb{R}^n$.
%for any $x_0\in\partial\Omega$, there exists an analytic function such that a finite union of graphs of some $n-1$ dimensional analytic functions in some suitable rectangle coordinate systems.
 The Buckling type equation with homogeneous boundary conditions is as follows:
\begin{equation}\label{basic equations}
\begin{cases}
\triangle^2u+\lambda\triangle u-k^2u=0,\quad in\quad \Omega,\\
u=u_{\nu}=0,\quad on\quad \partial\Omega,
\end{cases}
\end{equation}
where $\nu$ is the unit outer normal vector of $\partial\Omega$, and $u_{\nu}$ is the directional derivative along $\nu$. We also assume that $\lambda,\ k\geq0$. When $k=0$, it is the standard Buckling equation; when $\lambda=0$, it is the eigenvalue problem of the bi-Laplacian operator. The Buckling equation comes from the study of the vibration of beams and buckling of elastic structures and describes the critical buckling load of a clamped plate subjected to a uniform compressive force around its boundary $\cite{Payne}$. %介绍一下k和lambda的关系。有文章说明。简单引一下相关文献。对什么样的k和lambda，解存在；以及对k和lambda之间有什么定量的关系。我们只关心k和lambda较大的情况。

The present paper focuses on some quantitative properties including measure bounds of nodal sets, the unique continuation, and the smallness propagation of solutions. These properties for partial differential equations are crucial for understanding the growth, uniqueness, distribution of nodal sets and stability of solutions, and have been very important topics involving a large number of intensive studies in the past decades.
%Measure estimates of nodal sets and the unique continuation property of solutions to partial differential equations are important topics and highly correlated. 
One of the famous problems in this aspect is that, for any compact $C^{\infty}$ manifold without boundary, the upper and lower bounds of $(n-1)-$dimensional Huasdorff measure of nodal sets of eigenfunctions of the Laplacian operator both are comparable to $\sqrt{\lambda}$, where $\lambda$ is the corresponding eigenvalue. This problem nowadays is known as $Yau$'s conjecture $\cite{Yau}$.
There are various interesting results in this direction. When the manifold is analytic, the lower bound of this conjecture was proved for surfaces by J. Br\"uning in $\cite{Bruning}$, and  S.-T. Yau, independently $\cite{Logunov2}$. In $1988$, H. Donnelly and C. Fefferman in $\cite{Donnelly and Fefferman1}$ proved the conjecture for any dimensional analytic manifolds. In $1990$ in $\cite{Donnelly and Fefferman2}$, they also obtained that the maximal vanishing order of the eigenfunctions is no more than $C\sqrt{\lambda}$. In $1991$, F. H. Lin in $\cite{Lin}$ proved the monotonicity formula of the frequency function, established the measure upper bounds of nodal sets of solutions to some second-order linear and uniformly elliptic equations, and also derived the upper measure bound of the conjecture for any dimensional analytic manifolds. In 1990, H. Donnelly and C. Fefferman in $\cite{Donnelly and Fefferman2}$ obtained that, for any two dimensional $C^{\infty}$ manifold without boundary, the upper measure bound is $C\lambda^{\frac{3}{4}}$. It was improved by A. Logunov and  E. Malinnikova in $\cite{Logunov3}$ to $C\lambda^{\frac{3}{4}-\epsilon}$ for some positive constant $\epsilon$. In $1989$, R. Hardt and L. Simon studied the high dimensional $C^{\infty}$ case and showed that the upper measure bound is $\lambda^{C\sqrt{\lambda}}$. In $2018$, A. Logunov in $\cite{Logunov1}$ improved the result to $C\lambda^{\alpha}$ for some positive constant $\alpha>\frac{1}{2}$. In $\cite{Kukavica}$, I. Kukavika considered the linear and uniformly elliptic operator $\mathcal{A}$ of $2m-$order  with analytic coefficients and proved that, if the boundary $\partial\Omega$ is analytic, the upper measure bounds of nodal sets of solutions to the equation $\mathcal{A}u=\lambda u$ with analytic homogeneous boundary conditions are less than or equal to $C\lambda^{\frac{1}{2m}}$. In $2000$, Q. Han in $\cite{Han}$ described the structures of the nodal sets of solutions to the linear and uniformly elliptic equations of higher order. In $\cite{Tian and Yang2}$, the authors showed the upper measure bounds of nodal sets of eigenfunctions to the bi-Laplacian operator with non-analytic boundary data. In $\cite{Lin and Zhu}$, F. H. Lin and J. Zhu obtained upper bounds of nodal sets for eigenfunctions of eigenvalue problems including bi-harmonic Steklov eigenvalue problems, buckling eigenvalue problems and champed-plate eigenvalue problems by using analytic estimates of Morrey-Nirenberg and Carleman estimates. There are also various papers discussing the lower measure bounds of nodal sets of eigenfunctions, see for example $\cite{Colding, Logunov2, Sogge and Zelditch}$ and references therein. 

The unique continuation property has been a very active research topic in recent decades. N. Garofalo and F. H. Lin in $\cite{G and L1}$ and $\cite{G and L2}$ proved the monotonicity formula for the frequency functions, the doubling conditions of solutions to linear and uniformly elliptic equations of second order, and obtained the strong unique continuation property. In $1998$, I. Kukavica in $\cite{Kukavica2}$ gave an upper bound for the vanishing order of solutions of some second-order linear and uniformly elliptic equations. J. Zhu in $\cite{J.Y.Zhu3}$ obtained the doubling inequality and the vanishing order of the solutions to the bi-Laplacian equation. In $\cite{J.Y.Zhu2}$, he further gave a bound of the maximal vanishing order of solutions to higher-order elliptic equations with singular lower terms.  G. Alessandrini, L. Rondi, E. Rosset, and S. Vessella in $\cite{2009}$ established the three-spheres inequality and the stability for the Cauchy problem for elliptic equations. A. Logunov and E. Malinnikova in $\cite{Logunov and Milinicova}$ showed the quantitative propagation of smallness for solutions of elliptic equations. For various related results, see $\cite{LogunovMil, BJFA, Davey, Kenig, ZhuAJM}$.

The vanishing order of $u\in C^{\infty}(\Omega)$ at $x_0\in\Omega$ is the nonnegative integer $m$ such that
\begin{equation}\label{definition of vanishing order}
\begin{cases}
D^{\alpha}u(x_0)=0,\quad \forall\ \ |\alpha|<m,\\
D^{\alpha}u(x_0)\neq0,\quad for\ some\ |\alpha|=m,
\end{cases}
\end{equation}
where $\alpha=(\alpha_1,\cdots, \alpha_n)$ is a multi-index, each $\alpha_i$ is a nonnegative integer for any $i=1,2,\cdots,n$, and $D^{\alpha}u=D^{\alpha_1}_{x_1}D^{\alpha_2}_{x_2}\cdots D^{\alpha_n}_{x_n}u$. 
Moreover, if for any positive integer $m$, it holds that
\begin{equation}\label{definition of infinite order}
D^{\alpha}u(x_0)=0,\quad \forall\ \ |\alpha|<m,
\end{equation}
then we say that $u$ vanishes to infinite order at $x_0$.  
%讲一下定量唯一延拓性的定义。
The strong unique continuation property means that, if $u$ vanishes to infinite order at some point $x_0$, then $u\equiv0$ in the connected component containing $x_0$. %Moreover, the quantitative unique continuation property means that, if the vanishing order of $u$ is large than some given $M>0$, then $u\equiv0$ in the connected component containing $x_0$.%The quantitative unique continuation property means that, $u$ has the maximum vanishing order in $\Omega$, that is, if the vanishing order of $u$ at some point $x_0$ in $\Omega$ is large than the maximum vanishing order, then $u\equiv0$ in $\Omega$.

%{\color{red}%The present paper studies the solutions to the boundary value problem $(\ref{basic equations})$ of the buckling type equation. {\color{blue}Buckling problem comes from the vibration of the beam and has a strong background in mechanics. The quantitative uniqueness property, the measure of nodal sets, and the propagation of smallness are helpful in the study of the solutions to the buckling type equation.}
%We focus on the measure estimates of nodal sets and the maximal vanishing order of solutions. {\color{blue}We require that either $\lambda$ or $k$ is large enough.}}
%{\color{blue}and the upper measure bounds obtained in Theorem $\ref{nodal set}$ below is sharp.}
%在response中解释如何解析延拓。

The main results of this paper are the following three theorems. 

\begin{theorem}\label{quantitative uniqueness}
Assume that $\Omega$ is a bounded, connected and analytic domain of $\mathbb{R}^n$, and $k,\ \lambda\geq0$ and at least one of them large enough. Then, for a solution $u$ to $(\ref{basic equations})$, there exists a positive constant $C$ depending only on $n$ and $\Omega$, such that the maximal vanishing order of $u$ at any point $x\in\Omega$ is less than or equal to $C(\sqrt{\lambda}+\sqrt{k})$. In other words, if the vanishing order of $u$ at some point $x\in\Omega$ is larger than $C(\sqrt{\lambda}+\sqrt{k})$, then $u$ must be identically zero in $\Omega$.
\end{theorem}
%对大的k和lambda，提消失阶和测度估计才有意义。

\begin{theorem}\label{nodal set}
Let $u$ be a solution of $(\ref{basic equations})$, and $\Omega$ be a bounded and analytic domain. Then for $k, \lambda\geq0$, and at least one of them large enough, 
\begin{equation}
\mathcal{H}^{n-1}\left(\left\{x\in\Omega\ \big|\ u(x)=0\right\}\right)\leq C(\sqrt{\lambda}+\sqrt{k}),
\end{equation}
where $C$ is a positive constant depending only on $n$ and $\Omega$, and  $\mathcal{H}^n$ is the $n-$dimensional Hausdorff measure. 
\end{theorem}

\begin{theorem}\label{propagation of smallness}
Let $u$ be a solution of $(\ref{basic equations})$ in a bounded and connected domain $\Omega$. Assume that $G\subset\subset\Omega$ is a connected and open set, and  $E$ is a convex subset of $\Omega$ with $\mathcal{H}^n(E)\geq \epsilon$ for some positive constant $\epsilon$. If
\begin{equation*}
\|u\|_{L^{\infty}(E)}\leq\eta,\quad \|u\|_{L^{\infty}(\Omega)}\leq1,
\end{equation*}
then for $\lambda>0$ and $k>0$, at least one of them large enough, it holds that
\begin{equation}\label{propagation of smallness of G and E}
\|u\|_{L^{\infty}(G)}\leq e^{C({\sqrt{\lambda}+\sqrt{k}})}\eta^{\delta},
\end{equation}
where $C$ and $\delta$ are positive constants depending only on $n$, $diam(\Omega)$, $dist(G,\partial\Omega)$ and $\epsilon$.
\end{theorem}

In order to show the above results, we first explicitly establish a series of elliptic estimates involving $\lambda$ and $k$. With the help of introducing the frequency and doubling index related to solutions to the buckling type equation, we control the vanishing order and upper measure bounds of nodal sets of solutions by the frequency after deriving its monotonicity, doubling estimates, and mutually controlled relationship between it and the doubling index. We further show the measure upper bounds by the standard complexification. Finally, we establish the three sphere inequality and prove the quantitative propagation of smallness by iteration arguments. We point out that it is important for us to analytically extend the solutions considered to some neighborhood of $\Omega$ because of the analyticity of the solutions and $\partial\Omega$ in this paper. 

The rest of this paper is organized as follows. In the second section, we give the $L^2$ and $L^{\infty}$ estimates for every order derivative of $u$ in $\Omega$ explicitly involving $\lambda$ and $k$, and analytically extend $u$ across the boundary $\partial\Omega$. In the third section, we introduce the frequency and doubling index, and show an upper bound for the vanishing order and the quantitative unique continuation of $u$, i.e., proving Theorem $\ref{quantitative uniqueness}$. %{\color{red}We also give an example to illustrate the sharpness of Theorem $\ref{quantitative uniqueness}$.}
In the fourth section, we prove Theorem $\ref{nodal set}$ to give an upper measure bound for the nodal set of $u$ in $\Omega$.
Finally, in the fifth section, we prove Theorem $\ref{propagation of smallness}$, and show the propagation of smallness of $u$. In the rest of this paper, $C$ and $C'$ in different lines may be different positive constants depending only on $n$ and $\Omega$.
%%%%%%%%%%%%%%%%%%%%%%%%%%%%%%

\section{A priori estimates for any order derivatives of $u$}

This section will give the estimates of any order derivatives of a solution $u$ to $(\ref{basic equations})$. We first recall the following lemma which comes from $\cite{Lions and magenes}$. 

\begin{lemma}\label{basic theorem 3}
Let  $u\in\mathcal{D}(B_r^+(0)):=\cap_{m=0}^{\infty}W^{m,2}(B_r^+(0))$ and $D^{l}_nu=\frac{\partial^lu}{\partial x_n^l}$, where $W^{m,2}$ is the standard Sobolev space, $B_{r}^+(0)=\{ x\ |\ |x|<r,\ x_n>0\}$ is the upper half ball with radius $r$ centered at the origin. Then for any $0<\rho\leq r$, and any $\epsilon>0$, there exists a positive constant $C$ depending on $\epsilon$, $n$ and $r$, such that
\begin{equation}
\sum\limits_{t=1}^3\sum\limits_{|\alpha|=t,\alpha_n=0}\|D^{4-t}_nD^{\alpha}u\|_{L^2(B_{\rho}^+(0))}\leq \epsilon\|D_n^4u\|_{L^2(B_{\rho}^+(0))}+C\sum\limits_{|\alpha|=4,\alpha_n=0}\|D^{\alpha}u\|_{L^2(B_{\rho}^+(0))}.
\end{equation}
\end{lemma}

Next, we define 
\begin{equation*}
\bar{u}(x,x_{n+1})=u(x)e^{\sqrt{\frac{\lambda}{2}}x_{n+1}}.
\end{equation*}
Then $\bar{u}$ satisfies the following equation:
\begin{equation}\label{rewrite equation}
\triangle^2\bar{u}=\Lambda\bar{u}\quad in\quad\Omega\times\mathbb{R},
\end{equation}
with the boundary conditions below:
\begin{equation}\label{rewrite boundary conditions}
\bar{u}=0,\quad\bar{u}_{\nu}=0\quad on\quad\partial\Omega\times\mathbb{R}.
\end{equation}
Here $\Lambda=\frac{\lambda^2}{4}+k^2$.
% and $\bar{u}\in W^{2,2}(\Omega\times\mathbb{R})$. 
In the following, we always assume that $\Lambda>0$ is large enough.

\begin{remark}\label{analytic remark}
From the standard elliptic theory ($\cite{Lions and magenes}$, Chapter $8$),  the solutions to the problems $(\ref{basic equations})$ and $(\ref{rewrite equation})$ belong to $W^{m,2}$ for any positive integer $m$, and are analytic in $\Omega$.
\end{remark}

%{\color{red}From the classical PDE theory, one can get the $W^{m,2}$ estimate of the solution $\bar{u}$ for any positive integer $m$. But we need more accurate estimates involving $\lambda$ and $k$ explicitly.}
\begin{lemma}\label{interior estimate}
Let $\bar{u}$ satisfy the equation $(\ref{rewrite equation})$. Then for any $z_0=(x_0,0)$ with $x_0\in\Omega$ and $B_r(x_0)\subseteq\Omega$, any multi-index $\alpha$,%写出差值不等式的具体形式。
\begin{equation}\label{W42 estimate}
\|D^{\alpha}\bar{u}\|_{W^{4,2}(B_{\eta r}(z_0))}\leq C\left(\Lambda+\frac{1}{(1-\eta)^4r^4}\right)\|D^{\alpha}\bar{u}\|_{L^2(B_{r}(z_0))},
\end{equation}
for any $\eta\in (0,1)$. Here $B_{r}(z_0)\subseteq\Omega\times\mathbb{R}$ is the ball in $\mathbb{R}^{n+1}$ centered at $z_0$ with  its radius $r$, and $C$ is a positive constant depending only on $n$.
\end{lemma}

\begin{proof}

%Let $D^{h}_i\bar{u}(z)=\frac{\bar{u}(z+h e_i)-\bar{u}(z)}{h}$ be the difference quotient of $\bar{u}$ for $i\in \{1,...,n+1\}$, where $e_i$ is the unit vector of the $i-$th direction. 
Since $\bar{u}$ is real analytic by Remark $\ref{analytic remark}$, $\bar{u}_{ijml}=:D_{x_i}D_{x_j}D_{x_m}D_{x_l}u$ makes sense for any $i,j,m,l\in\{1,2,\cdots,n+1\}$. We multiply both sides of the equation $(\ref{rewrite equation})$ by $\bar{u}_{mmll}\psi$, and take integral over $\Omega\times\mathbb{R}$, here $\psi=\phi^4$, $\phi\in C^{\infty}(B_r(z_0))$, and
\begin{equation}
\begin{cases}
\phi(x)=1\quad  in \quad B_{\eta r}(z_0),\\
\phi(x)=0\quad outside \quad B_{\frac{1+\eta}{2}r}(z_0),\\
|D\phi(x)|\leq\frac{C}{(1-\eta)r},
\end{cases}
\end{equation}
 for some positive constant $C$ depending only on $n$. Then by integrating by parts, summing over $m,l$ from $1$ to $n+1$, we have for any $\epsilon>0$,
\begin{eqnarray*}
&&\Lambda\sum\limits_{m,l=1}^{n+1}\int_{B_{\frac{1+\eta}{2}r}(z_0)}\bar{u}\bar{u}_{mmll}\psi dz=\sum\limits_{i,j,m,l=1}^{n+1}\int_{B_{\frac{1+\eta}{2} r}(z_0)}\bar{u}_{iijj}\bar{u}_{mmll}\psi dz
\\&=&\sum\limits_{i,j,m,l=1}^{n+1}\int_{B_{\frac{1+\eta}{2} r}(z_0)}\bar{u}_{ijml}^2\psi dz-\sum\limits_{i,j,m,l=1}^{n+1}\int_{B_{\frac{1+\eta}{2}r}(z_0)}\bar{u}_{ijj}\bar{u}_{mmll}\psi_idz+\sum\limits_{i,j,m,l=1}^{n+1}\int_{B_{\frac{1+\eta}{2}r}(z_0)}\bar{u}_{ijj}\bar{u}_{mlli}\psi_mdz
\\&-&\sum\limits_{i,j,m,l=1}^{n+1}\int_{B_{\frac{1+\eta}{2}r}(z_0)}\bar{u}_{imj}\bar{u}_{mlli}\psi_jdz+\sum\limits_{i,j,m,l=1}^{n+1}\int_{B_{\frac{1+\eta}{2}r}(z_0)}\bar{u}_{imj}\bar{u}_{milj}\psi_ldz
\\&\geq&(1-\epsilon)\sum\limits_{i,j,m,l=1}^{n+1}\int_{B_{\frac{1+\eta}{2} r}(z_0)}\bar{u}_{ijml}^2\psi dz-\frac{C}{\epsilon}\sum\limits_{i,j,m=1}^{n+1}\int_{B_{\frac{1+\eta}{2} r}(z_0)}\bar{u}_{ijm}^2\frac{|D\psi|^2}{\psi} dz.
\end{eqnarray*}
In the last inequality, we have used the H\"older's inequality and Young inequality.
On the other hand, 
\begin{eqnarray*}
\Lambda\sum\limits_{m,l=1}^{n+1}\int_{B_{\frac{1+\eta}{2}r}(z_0)}\bar{u}\bar{u}_{mmll}\psi dz
\leq \frac{\Lambda^2}{\epsilon}\int_{B_{\frac{1+\eta}{2}r}(z_0)}\bar{u}^2\psi dz+\epsilon\sum\limits_{m,l=1}^{n+1}\int_{B_{\frac{1+\eta}{2}r}(z_0)}\bar{u}^2_{mmll}\psi dz.
\end{eqnarray*}
So
\begin{equation}\label{from 4 to 3}
\sum\limits_{i,j,m,l=1}^{n+1}\int_{B_r(z_0)}|\bar{u}_{ijml}|^2\phi^4 dz\leq\frac{C}{(1-\eta)^2r^2\epsilon}\sum\limits_{i,j,m=1}^{n+1}\int_{B_r(z_0)}|\bar{u}_{ijm}|^2\phi^2dz+\frac{C\Lambda^2}{\epsilon}\int_{B_r(z_0)}|\bar{u}|^2\phi^4dz,
\end{equation}
%where $C$ in different lines are distinct positive constants and depend only on $n$.
 So by choosing $\epsilon=\frac{1}{2}$, we have
\begin{equation}\label{four order estimate}
\sum\limits_{i,j,m,l=1}^{n+1}\int_{B_{\frac{1+\eta}{2}r}(z_0)}|\bar{u}_{ijml}|^2\phi^4dz\leq \frac{C}{(1-\eta)^2r^2}\sum\limits_{i,j,m=1}^{n+1}\int_{B_{\frac{1+\eta}{2}r}(z_0)}|\bar{u}_{ijm}|^2\phi^2dz+C\Lambda^2\int_{B_{\frac{1+\eta}{2}r}(z_0)}|\bar{u}|^2\phi^4dz.
\end{equation}

Now we consider the first term on the right hand side of $(\ref{four order estimate})$. In fact, by the direct calculation, integrating by parts, and the equation $(\ref{rewrite equation})$, we have
\begin{eqnarray*}
I_1&=:&\sum\limits_{i,j,m=1}^{n+1}\int_{B_{\frac{1+\eta}{2}r}(z_0)}|\bar{u}_{ijm}|^2\phi^2dz=-\sum\limits_{i,j,m=1}^{n+1}\int_{B_{\frac{1+\eta}{2}r}(z_0)}\bar{u}_{ij}\bar{u}_{ijmm}\phi^2dz-2\sum\limits_{i,j,m=1}^{n+1}\int_{B_{\frac{1+\eta}{2}r}(z_0)}\bar{u}_{ij}\bar{u}_{ijm}\phi\phi_mdz
\\&=&\sum\limits_{i,j,m=1}^{n+1}\int_{B_{\frac{1+\eta}{2}r}(z_0)}\bar{u}_{ijj}\bar{u}_{imm}\phi^2dz+2\sum\limits_{i,j,m=1}^{n+1}\int_{B_{\frac{1+\eta}{2}r}(z_0)}\bar{u}_{ij}\bar{u}_{imm}\phi\phi_jdz-2\sum\limits_{i,j,m=1}^{n+1}\int_{B_{\frac{1+\eta}{2}r}(z_0)}\bar{u}_{ij}\bar{u}_{ijm}\phi\phi_mdz
\\&=&-\sum\limits_{i,j,m=1}^{n+1}\int_{B_{\frac{1+\eta}{2}r}(z_0)}\bar{u}_{jj}\bar{u}_{iimm}\phi^2dz-2\sum\limits_{i,j,m=1}^{n+1}\int_{B_{\frac{1+\eta}{2}r}(z_0)}\bar{u}_{jj}\bar{u}_{imm}\phi\phi_i dz\\&+&2\sum\limits_{i,j,m=1}^{n+1}\int_{B_{\frac{1+\eta}{2}r}(z_0)}\bar{u}_{ij}\bar{u}_{imm}\phi\phi_jdz-2\sum\limits_{i,j,m=1}^{n+1}\int_{B_{\frac{1+\eta}{2}r}(z_0)}\bar{u}_{ij}\bar{u}_{ijm}\phi\phi_mdz
\\&=&-\Lambda\int_{B_{\frac{1+\eta}{2}r}(z_0)}\triangle\bar{u}\bar{u}\phi^2dz-2\sum\limits_{i,j,m=1}^{n+1}\int_{B_{\frac{1+\eta}{2}r}(z_0)}\bar{u}_{jj}\bar{u}_{imm}\phi\phi_i dz\\&+&2\sum\limits_{i,j,m=1}^{n+1}\int_{B_{\frac{1+\eta}{2}r}(z_0)}\bar{u}_{ij}\bar{u}_{imm}\phi\phi_jdz-2\sum\limits_{i,j,m=1}^{n+1}\int_{B_{\frac{1+\eta}{2}r}(z_0)}\bar{u}_{ij}\bar{u}_{ijm}\phi\phi_mdz
\\&\leq&\frac{1}{2}\int_{B_{\frac{1+\eta}{2}r}(z_0)}(\triangle\bar{u})^2\phi^2dz+\frac{\Lambda^2}{2}\int_{B_{\frac{1+\eta}{2}r}(z_0)}\bar{u}^2\phi^2dz
\\&+&\frac{3}{\epsilon_1}\int_{B_{\frac{1+\eta}{2}r}(z_0)}(\triangle\bar{u})^2|D\phi|^2dz+\frac{\epsilon_1}{3}\int_{B_{\frac{1+\eta}{2}r}(z_0)}|D\triangle\bar{u}|^2\phi^2dz
\\&+&\frac{3}{\epsilon_1}\sum\limits_{i,j=1}^{n+1}\int_{B_{\frac{1+\eta}{2}r}(z_0)}(\bar{u}_{ij})^2|D\phi|^2dz+\frac{\epsilon_1}{3}\int_{B_{\frac{1+\eta}{2}r}(z_0)}|D\triangle\bar{u}|^2\phi^2dz
\\&+&\frac{3}{\epsilon_1}\sum\limits_{i,j=1}^{n+1}\int_{B_{\frac{1+\eta}{2}r}(z_0)}(\bar{u}_{ij})^2|D\phi|^2dz+\frac{\epsilon_1}{3}\sum\limits_{i,j,m=1}^{n+1}\int_{B_{\frac{1+\eta}{2}r}(z_0)}|\bar{u}_{ijm}|^2\phi^2dz
\\&\leq&\epsilon_1\sum\limits_{i,j,m=1}^{n+1}\int_{B_{\frac{1+\eta}{2}r}(z_0)}|\bar{u}_{ijm}|^2\phi^2dz+\frac{C}{(1-\eta)^2r^2\epsilon_1}\sum\limits_{i,j=1}^{n+1}\int_{B_{\frac{1+\eta}{2}r}(z_0)}|\bar{u}_{ij}|^2dz+C\Lambda^2\int_{B_{\frac{1+\eta}{2}r}(z_0)}\bar{u}^2\phi^2dz,
\end{eqnarray*}
for any $\epsilon_1\in(0,1)$. Here $C$ is a positive constant depending only on $n$. So by choosing $\epsilon_1=\frac{1}{2}$, we obtain
\begin{equation}\label{three order estimate}
I_1\leq \frac{C}{(1-\eta)^2r^2}\int_{B_{\frac{1+\eta}{2}r}(z_0)}|D^2\bar{u}|^2dz+C\Lambda^2\int_{B_{\frac{1+\eta}{2}r}(z_0)}\bar{u}^2dz.
\end{equation}

Next, we estimate the first term of $(\ref{three order estimate})$. In fact, let $\bar{\phi}$ be a $C_0^{\infty}$ cut-off function such that $\bar{\phi}(z)=1$ when $|z-z_0|<\frac{1+\eta}{2}r$, $\bar{\phi}(z)=0$ when $|z-z_0|>r$, $0\leq\bar{\phi}\leq1$, $|D\bar{\phi}|<\frac{C}{(1-\eta)r}$, and $|D^2\bar{\phi}|\leq \frac{C}{(1-\eta)^2r^2}$. Then define
\begin{equation}
\psi=
\begin{cases}
e^{1-\bar{\phi}^{-1}},\quad 0<\bar{\phi}\leq1,\\
0,\quad\bar{\phi}=0.
\end{cases}
\end{equation}
Thus $\psi$ satisfies that for any $l>0$,
\begin{equation*}
\lim\limits_{\bar{\phi}\rightarrow0+}\frac{\psi}{\bar{\phi}^l}=0.
\end{equation*}
Moreover, through some direct calculations,
\begin{equation}\label{derivative of psi}
\begin{cases}
D\psi=\psi\frac{D\bar{\phi}}{\bar{\phi}^2},\\
\triangle\psi=\psi\left(\frac{|D\bar{\phi}|^2}{\bar{\phi}^4}-2\frac{|D\bar{\phi}|^2}{\bar{\phi}^3}+\frac{\triangle\bar{\phi}}{\bar{\phi}^2}\right).
\end{cases}
\end{equation}
By multiplying $\bar{u}\psi$ on both sides of $(\ref{rewrite equation})$ and using integration by parts, we have
\begin{eqnarray*}
\Lambda\int_{B_{r}(z_0)}\bar{u}^2\psi dz&=&\int_{B_{r}(z_0)}\triangle^2\bar{u}\bar{u}\psi dz
\\&=&\int_{B_{r}(z_0)}|\triangle\bar{u}|^2\psi dz+2\int_{B_{r}(z_0)}\triangle\bar{u}D\bar{u}D\psi dz+\int_{B_{r}(z_0)}\triangle\bar{u}\bar{u}\triangle\psi dz
\\&=&\int_{B_{r}(z_0)}|\triangle\bar{u}|^2\psi dz+2\int_{B_{r}(z_0)}\triangle\bar{u}D\bar{u}\psi\frac{D\bar{\phi}}{\bar{\phi}^2}dz+\int_{B_{r}(z_0)}\triangle\bar{u}\bar{u}\psi\left(\frac{|D\bar{\phi}|^2}{\bar{\phi}^4}-2\frac{|D\bar{\phi}|^2}{\bar{\phi}^3}+\frac{\triangle\bar{\phi}}{\bar{\phi}^2}\right)dz.
\end{eqnarray*}
So for any $\epsilon_2\in(0,1)$,
\begin{eqnarray*}
\int_{B_{r}(z_0)}|\triangle\bar{u}|^2\psi dz&\leq&\Lambda\int_{B_{r}(z_0)}|\bar{u}|^2\psi dz+\epsilon_2\int_{B_{r}(z_0)}|\triangle\bar{u}|^2\psi dz\\&+&\frac{C}{\epsilon_2}\int_{B_{r}(z_0)}|D\bar{u}|^2\psi\frac{|D\bar{\phi}|^2}{\bar{\phi}^4}dz+\frac{C}{\epsilon_2}\int_{B_{r}(z_0)}\bar{u}^2\psi\frac{|D\bar{\phi}|^4+|\triangle\bar{\phi}|^2}{\bar{\phi}^8}dz.
\end{eqnarray*}
Choosing $\epsilon_2=\frac{1}{2}$, we have
\begin{equation}
\int_{B_{r}(z_0)}|\triangle\bar{u}|^2\psi dz\leq\frac{C}{(1-\eta)^2r^2}\int_{B_{r}(z_0)}|D\bar{u}|^2\frac{\psi}{\bar{\phi}^4}dz+C\left(\Lambda+(1-\eta)^{-4}r^{-4}\right)\int_{B_{r}(z_0)}\bar{u}^2\frac{\psi}{\bar{\phi}^8}dz.
\end{equation}
Since
\begin{equation}
\int_{B_{r}(z_0)}|D^2\bar{u}|^2\psi dz\leq \int_{B_{r}(z_0)}|\triangle\bar{u}|^2\psi dz+\frac{C}{(1-\eta)^2r^2}\int_{B_{r}(z_0)}|D\bar{u}|^2\frac{\psi}{\bar{\phi}^4}dz,
\end{equation}
which comes from the integration by parts, we have
\begin{equation}\label{second order estimate}
I_2=:\int_{B_{r}(z_0)}|D^2\bar{u}|^2\psi dz\leq\frac{C}{(1-\eta)^2r^2}\int_{B_{r}(z_0)}|D\bar{u}|^2\frac{\psi}{\bar{\phi}^4}dz+C\left(\Lambda+(1-\eta)^{-4}r^{-4}\right)\int_{B_{r}(z_0)}\bar{u}^2\frac{\psi}{\bar{\phi}^8}dz.
\end{equation}
Integration by parts again, for any $\epsilon_3,\ \epsilon_4>0$, 
\begin{eqnarray*}
I_3&=:&\int_{B_{r}(z_0)}|D\bar{u}|^2\frac{\psi}{\bar{\phi}^4}dz=-\int_{B_{r}(z_0)}\bar{u}\triangle\bar{u}\frac{\psi}{\bar{\phi}^4}dz-\int_{B_{r}(z_0)}\bar{u}D\bar{u}D\left(\frac{\psi}{\bar{\phi}^4}\right)dz
\\&\leq&\epsilon_3\int_{B_{r}(z_0)}|\triangle\bar{u}|^2\psi dz+\frac{C}{\epsilon_3}\int_{B_{r}(z_0)}\bar{u}^2\frac{\psi}{\bar{\phi}^8}dz+\epsilon_4\int_{B_{r}(z_0)}|D\bar{u}|^2\frac{\psi}{\bar{\phi}^4} dz+\frac{C}{(1-\eta)^2r^2\epsilon_4}\int_{B_{r}(z_0)}\bar{u}^2\frac{\psi}{\bar{\phi}^{12}}dz.
\end{eqnarray*}
Then by choosing $\epsilon_4=\frac{1}{2}$ and $\epsilon_3=\frac{(1-\eta)^2r^2}{8nC}$, where $C$ is the same positive constant as in $(\ref{second order estimate})$, we have
\begin{equation}\label{two order estimate}
\int_{B_{\frac{1+\eta}{2}r}(z_0)}|D^2\bar{u}|^2dz\leq C\left(\Lambda+(1-\eta)^{-4}r^{-4}\right)\int_{B_{r}(z_0)}\bar{u}^2dz.
\end{equation}
From the inequalities $(\ref{four order estimate})$, $(\ref{three order estimate})$ and $(\ref{two order estimate})$, we have
\begin{equation}
\|\bar{u}\|_{W^{4,2}(B_{\eta r}(z_0))}^2\leq C(\Lambda^2+(1-\eta)^{-8}r^{-8})\int_{B_{r}(z_0)}|\bar{u}|^2dz.
\end{equation}
Then from the fact that $\bar{u}(z)=\bar{u}(x,x_{n+1})=u(x)e^{\sqrt{\frac{\lambda}{2}}x_{n+1}}$, the case $|\alpha|=0$ of $(\ref{W42 estimate})$ is obtained.

Since for any multi-index $\alpha$, 
\begin{equation*}
\triangle^2D^{\alpha}\bar{u}-\Lambda D^{\alpha}\bar{u}=0,\quad in\quad\Omega\times\mathbb{R},
\end{equation*}
the desired result is obtained by applying the above argument to $D^{\alpha}\bar{u}$ and the fact that $\bar{u}(z)=\bar{u}(x,x_{n+1})=u(x)e^{\sqrt{\frac{\lambda}{2}}x_{n+1}}$.
\end{proof}

\begin{remark}\label{interior remark}
From the Sobolev interpolation inequality, for any $\epsilon>0$ and any $u\in W^{4,2}(B_r(z_0))$,
\begin{equation}\label{interpolation}
\|u\|_{W^{1,2}(B_{r}(z_0))}\leq \epsilon\|u\|_{W^{4,2}(B_r(z_0))}+C\epsilon^{-1/3}\|u\|_{L^2(B_r(z_0))}.
\end{equation}
Then from $(\ref{interpolation})$ with $\epsilon=\left(\Lambda^{1/4}+\frac{1}{(1-\eta)r}\right)^{-3}$ and Lemma $\ref{interior estimate}$, for any $\eta\in(0,1)$ and any $r>0$ such that $dist(x_0,\partial\Omega)>r$, 
\begin{eqnarray}
\|\bar{u}\|_{W^{1,2}(B_{\eta r}(z_0))}\leq C\left(\Lambda^{1/4}+\frac{1}{(1-\eta)r}\right)\|\bar{u}\|_{L^2(B_r(z_0))}.
\end{eqnarray}
So by the iteration argument,
\begin{eqnarray}
\|\bar{u}\|_{W^{m,2}(B_{\eta r}(z_0))}&\leq&C\left(\Lambda^{1/4}+\frac{m}{(1-\eta)r}\right)\|\bar{u}\|_{W^{m-1,2}(B_{\left(\eta+\frac{1-\eta}{m}\right)r}(z_0))}\nonumber
\\&\leq&C^2\left(\Lambda^{1/4}+\frac{m}{(1-\eta)r}\right)^2\|\bar{u}\|_{W^{m-2,2}(B_{\left(\eta+2\frac{1-\eta}{m}\right)r}(z_0))}\nonumber
\\&\leq&\cdots\nonumber
\\&\leq&C^m\left(\Lambda^{1/4}+\frac{m}{(1-\eta)r}\right)^m\|\bar{u}\|_{L^2(B_r(z_0))},
\end{eqnarray}
where $z_0=(x_0,0)$, and $C$ is a positive constant depending only on $n$.
\end{remark}

Now we focus on the boundary estimates. In fact, for any fixed point $x_0\in\partial\Omega$, we locally flat the boundary $\partial\Omega$ near $x_0$ by the following process. Without loss of generality, assume that $x_0=0$, and for $r$ small enough, the set $\Omega\cap B_r(0)$ can be expressed in the following form:
\begin{equation*}
\Omega\cap B_r(x_0)=\left\{x\in B_r(x_0)\ |\ x_n>\gamma(x_1,\cdots,x_{n-1})\right\}.
\end{equation*}
where $\gamma$ is a real analytic function. Then define $z=(x,x_{n+1})=(x_1,\cdots,x_n,x_{n+1})$, and
\begin{equation}\label{equations changing coordinates}
\begin{cases}
y_i=\Phi_i(z):=x_i,\quad i=1,2,\cdots,n-1,\\
y_n=\Phi_n(z):=x_n-\gamma(x_1,\cdots,x_{n-1}),\\
y_{n+1}=x_{n+1}.
\end{cases}
\end{equation}
It is also denoted by $y=\Phi(z)$. Similarly, we write $z=\Psi(y)$ with $\Psi=\Phi^{-1}$. Then the map $y=\Phi(z)$ straightens $\partial\Omega$ near $0$, and pushes $\Omega\cap B_r(0)$ to $\Phi(\Omega\cap B_r(0))\subseteq B_R^+(x_0,0)$ for some $r,R>0$, and the map $\Phi$ is real analytic. Under this transformation, $\bar{u}(z)$ and $\bar{v}(z)=\triangle \bar{u}(z)$ become $\bar{\bar{u}}(y)$ and $\bar{\bar{v}}(y)$, i.e., $\bar{\bar{u}}(y)=\bar{u}(\Psi(y))$ and $\bar{\bar{v}}(y)=\bar{v}(\Psi(y))$, respectively, which satisfy the following equations:
\begin{equation}\label{basic equations after changing coordinates1}
\begin{cases}
\mathcal{L}\bar{\bar{u}}=a_{ij}(y)\bar{\bar{u}}_{ij}(y)+b_i(y)\bar{\bar{u}}_i(y)=\bar{\bar{v}}(y),\\
\mathcal{L}\bar{\bar{v}}=a_{ij}(y)\bar{\bar{v}}_{ij}(y)+b_i(y)\bar{\bar{v}}_i(y)=\Lambda\bar{\bar{u}}(y),
\end{cases}
\end{equation}
in the domain $\Phi(\Omega\cap B_r(0))\times(-\infty,+\infty)$. 
Here the coefficients $a_{ij}(y)=\sum\limits_{m=1}^n\Phi_{x_m}^i(\Psi(y))\Phi^j_{x_m}(\Psi(y))$, and $b_i(y)=\sum\limits_{j=1}^n\Phi^i_{x_jx_j}(\Psi(y))$. The coefficients $a_{ij}$ and $b_i$ are also analytic, and the operator $\mathcal{L}$ is uniformly elliptic. The map $\Phi$ can be chosen such that
\begin{equation}\label{coefficient condition}
a_{ij}(0)=\delta_{ij}, \quad |a_{ij}(y)-a_{ij}(0)|\leq C_0|y|,
\end{equation}
for any $i,j=1,\cdots,n+1$. Here $C_0$ is a positive constant depending only on $n$ and $\Omega$. Then for $r$ small enough and some positive constant $\delta_0$, the matrix $\{a_{ij}\}$ satisfies that $a_{ij}(y)\xi_i\xi_j\geq \delta_0|\xi|^2$ for any $y\in B_r^+(0)\subseteq\mathbb{R}^{n+1}$ and $\xi\in \mathbb{R}^{n+1}$. Moreover, the boundary condition becomes
\begin{equation}\label{boundary conditions after changing coordinates}
\bar{\bar{u}}=0,\quad
\bar{\bar{u}}_n=0,
\end{equation}
on $\Gamma^*$, the flat boundary of $B_r^+(0)\subseteq\mathbb{R}^{n+1}$.

Now we use this transformation to establish the boundary estimate below.
\begin{lemma}\label{boundary estimate1}
There exist positive constants $\rho_0$ and $L_0$ depending only on $n$, $\Omega$, and $\mathcal{L}$, such that for any integer $m$ and some $\rho<\rho_0$,
\begin{equation}
\sup\limits_{x\in B_{\rho}^+(0)}|D^{\alpha}\bar{\bar{u}}|\leq m!L_0^me^{\Lambda^{1/4}}\|\bar{\bar{u}}\|_{L^2(B_{\rho_0}^+(0))},\quad\forall\ |\alpha|=m.
\end{equation}
\end{lemma}

\begin{proof}
By the above transformation, there exists a positive constant $\rho_0$ depending only on $n$ and $\Omega$, such that the function $\bar{\bar{u}}$ satisfies the equation $\mathcal{L}^2\bar{\bar{u}}=\Lambda\bar{\bar{u}}$ in $B_{\rho_0}^+(0)$, and $\bar{\bar{u}}=\bar{\bar{u}}_n=0$ on $\Gamma^*$. Then
\begin{equation}
\|(\mathcal{L}^2)^l\bar{\bar{u}}\|_{L^2(B_{\rho_0}^+(0))}=\Lambda^l\|\bar{\bar{u}}\|_{L^2(B_{\rho_0}^+(0))}\leq (4l)!e^{\Lambda^{1/4}}\|\bar{\bar{u}}\|_{L^2(B_{\rho_0}^+(0))}.
\end{equation}
For the operator $\mathcal{L}^2$ and the function $\frac{\bar{\bar{u}}}{e^{\Lambda/4}\|\bar{\bar{u}}\|_{L^2(B_{\rho_0}^+)}}$, we use Theorem $1.3$ in Chapter $8$ in $\cite{Lions and magenes}$ to obtain that, there exist positive constants $\rho$ and $L_0$ depending only on $n$, $\Omega$ and $z_0$, such that for any positive integer $m$,
\begin{equation}
\sup\limits_{x\in B_{\rho}^+(0)}|D^{\alpha}\bar{\bar{u}}|\leq m!L_0^me^{\Lambda^{1/4}}\|\bar{\bar{u}}\|_{L^2(B_{\rho_0}^+(0))},\quad\forall\ |\alpha|=m.
\end{equation}
\end{proof}

\begin{lemma}\label{explain theorem}
Let $u$ be a solution to the problem $(\ref{basic equations})$. Then there exists a positive constant $R$ depending only on $n$ and $\partial\Omega$, such that $u$ can be extended into the neighborhood $\Omega_R:=\left\{x\in\mathbb{R}^n\ |\ dist(x,\Omega)<R\right\}$ with
\begin{equation}
\|u\|_{L^{\infty}(\Omega_R)}\leq e^{C\Lambda^{\frac{1}{4}}}\|u\|_{L^2(\Omega)},
\end{equation}
where $C$ and $R$ are positive constants depending only on $n$ and $\Omega$.
\end{lemma}

\begin{proof}
By Lemma $\ref{boundary estimate1}$ and the finite covering, there exist constants $r_0>0$, $\tau>1$, and $L_0>1$ depending only on $n$ and $\Omega$, such that for any positive integer $m$,
\begin{equation}\label{boundary L infinity}
\sup\limits_{T_{r_0}(\partial\Omega)\times(-r_0,r_0)}|D^{\alpha}\bar{u}|\leq m!L_0^me^{\Lambda^{1/4}}\|\bar{u}\|_{L^2(\Omega\times(-\tau r_0,\tau r_0))},\quad\forall\ |\alpha|=m,
\end{equation}
where $T_{r_0}(\partial\Omega)=\{x\in\Omega\ |\ dist(x,\partial\Omega)<r_0\}$.
From $(\ref{boundary L infinity})$ and the fact that $\bar{u}(x,x_{n+1})=u(x)e^{\Lambda^{1/4}x_{n+1}}$, for any $x\in T_{r_0}(\partial\Omega)$,
\begin{equation}
|D^{\alpha}u(x)|\leq  m!L_0^me^{C\Lambda^{1/4}}\|u\|_{L^2(\Omega)} ,\quad\forall\ |\alpha|=m.
\end{equation}
Since $\partial\Omega$ is compact, one can extend $u(x)$ analytically into a neighborhood of $\Omega$, denoted by $\Omega_R=\{x\in\mathbb{R}^n\ |\ dist(x,\Omega)<R\}$. Here $R$ is a positive constant depending only on $\partial\Omega$ and $n$. In fact, for any $x_0\in T_{r_0}(\partial\Omega)$ and any $x\in B_{R}(x_0)$, there holds
\begin{equation*}
u(x)=\sum\limits_{|\alpha|=0}^{\infty}\frac{1}{\alpha!}D^{\alpha}u(x_0)(x-x_0)^{\alpha}.
\end{equation*}
That is the Taylor series of $u(x)$. So by choosing $R\leq cL_0^{-1}$, where $c$ is a positive constant depending only on $n$,

\begin{eqnarray}\label{boundary L infinity2}
|u(x)|&\leq&\sum\limits_{|\alpha|=0}^{\infty}\frac{1}{\alpha!}|D^{\alpha}u(x_0)||x-x_0|^{|\alpha|}\leq e^{C\Lambda^{1/4}}\|u\|_{L^2(\Omega)}\sum\limits_{|\alpha|=0}^{\infty}\frac{|\alpha|!(RL_0)^{|\alpha|}}{\alpha!}\nonumber
\\&\leq&e^{C\Lambda^{1/4}}\|u\|_{L^2(\Omega)}\sum\limits_{m=0}^{\infty}(RL_0)^m\sum\limits_{|\alpha|=m}\frac{|\alpha|!}{\alpha!}\nonumber
\\&\leq&e^{C\Lambda^{1/4}}\|u\|_{L^2(\Omega)}\sum\limits_{m=0}^{\infty}(nRL_0)^m\nonumber
\\&\leq&C'e^{C\Lambda^{1/4}}\|u\|_{L^2(\Omega)}.
\end{eqnarray}

%{\color{blue}
%\begin{eqnarray}\label{boundary L infinity2}
%|u(x)|&\leq&\sum\limits_{|\alpha|=0}^{\infty}\frac{1}{\alpha!}|D^{\alpha}u(x_0)||x-x_0|^{|\alpha|}
%\leq e^{C\Lambda^{1/4}}\|u\|_{L^2(\Omega)}\sum\limits_{|\alpha|=0}^{\infty}\frac{|\alpha|!(RL_0)^{|\alpha|}}{\alpha!}}\nonumber
%\\&\leq&e^{C\Lambda^{1/4}}\|u\|_{L^2(\Omega)}
%\sum\limits_{m=0}^{\infty}(RL_0)^m\sum\limits_{|\alpha|=m}\frac{|\alpha|!}{\alpha!}\nonumber
%\\&\leq&C'e^{C\Lambda^{1/4}}\|u\|_{L^2(\Omega)}.
%\end{eqnarray}}
Here we used the fact that $n^m=(1+1+\cdots+1)^m=\sum\limits_{|\alpha|=m}\frac{|\alpha|!}{\alpha!}$.

On the other hand, from Sobolev Embedding Theorem and Remark $\ref{interior remark}$ with $\eta=1/2$ and $r=(C\Lambda^{1/4})^{-1}\leq r_0$,  for any $z=(x,0)$ with $x\in\Omega\setminus T_{r_0}(\partial\Omega)$,
\begin{eqnarray}\label{interior L infinity}
|\bar{u}(z)|&\leq&\|\bar{u}\|_{L^{\infty}(B_{\eta r}(z))}\leq C(\eta r)^{-(n+1)/2}\|\bar{u}\|_{W^{l,2}(B_{\eta r}(z))}\nonumber
\\&\leq& C^{l}(l)^{l}\Lambda^{\frac{l}{4}+\frac{n+1}{8}}\|\bar{u}\|_{L^2(B_r(z))} \leq e^{C\log\Lambda}\|\bar{u}\|_{L^2(\Omega\times(-2r_0,2r_0))},\quad\forall\ |\alpha|=m,
\end{eqnarray}
where 
%$C$ in different terms are different positive constants depending only on $n$, and 
$l=\left[\frac{n+1}{2}\right]+1$. From $(\ref{interior L infinity})$ and the fact that $\bar{u}(x,x_{n+1})=u(x)e^{\Lambda^{1/4}x_{n+1}}$ again, for any $x\in \Omega\setminus T_{r_0}(\partial\Omega)$,
\begin{equation}\label{interior L infinity2}
|u(x)|=|\bar{u}(x,0)|\leq e^{C\log\Lambda}\|\bar{u}\|_{L^2(\Omega\times(-2r_0,2r_0))}\leq e^{C(\Lambda^{1/4}+\log\Lambda)}\|u\|_{L^2(\Omega)}\leq e^{C\Lambda^{1/4}}\|u\|_{L^2(\Omega)}.
\end{equation}
The inequalities $(\ref{boundary L infinity2})$ and $(\ref{interior L infinity2})$ complete the proof.
\end{proof}

\section{Quantitative Unique Continuation Property}
In this section, we will give the upper bound for the vanishing order of $u$ in $\Omega$. From Section $2$, $u$ is analytic in $\Omega_R$, so are $\triangle u$ and $\triangle^2u$. Therefore, by the uniqueness of the analytic continuation, the equation of $(\ref{basic equations})$ holds in $\Omega_R$.
Then we rewrite it in $\Omega_R$ as follows. Let $\widetilde{u}(x,x_{n+1})=u(x)e^{\sqrt{\frac{\lambda+\mu}{2}} x_{n+1}}$ and $\widetilde{v}(x,x_{n+1})=\left(\triangle u(x)+\frac{\lambda+\mu}{2}u(x)\right)e^{\sqrt{\frac{\lambda+\mu}{2}}x_{n+1}}$ with $\mu=\sqrt{\lambda^2+4k^2}$.
Then $\widetilde{u}$ satisfies that
\begin{equation}\label{new equations with mu}
\begin{cases}
\triangle\widetilde{u}=\widetilde{v},\\
\triangle\widetilde{v}=\mu\widetilde{v},
\end{cases}
\end{equation}
in $\Omega_R\times\mathbb{R}$.
So we define the frequency function and doubling index as follows.
\begin{definition}\label{definition of frequency and doubling index}
Let $\widetilde{u}(z)=\widetilde{u}(x,x_{n+1})=u(x)e^{\sqrt{\frac{\lambda+\mu}{2}} x_{n+1}}$ as above. Then $\widetilde{u}$ satisfies the equation $(\ref{new equations with mu})$. For $z_0=(x_0,0)$
we call the following quantities
\begin{equation}
N(z_0,r)=r\frac{\int_{B_r(z_0)}\left(|D\widetilde{u}|^2+|D\widetilde{v}|^2+\widetilde{u}\widetilde{v}+\mu|\widetilde{v}|^2\right)dz}{\int_{\partial B_r(z_0)}\left(\widetilde{u}^2+\widetilde{v}^2\right)d\sigma}=r\frac{\int_{\partial B_r(z_0)}\left(\widetilde{u}\widetilde{u}_{\nu}+\widetilde{v}\widetilde{v}_{\nu}\right)d\sigma}{\int_{\partial B_r(z_0)}\left(\widetilde{u}^2+\widetilde{v}^2\right)d\sigma},
\end{equation}
the frequency function with radius $r$ centered at $z_0$, and
\begin{equation}
M(z_0,r)=\frac{1}{2}log_2\left(\frac{\|\widetilde{u}\|^2_{L^{\infty}(B_r(z_0))}+\|\widetilde{v}\|^2_{L^{\infty}(B_r(z_0))}}{\|\widetilde{u}\|^2_{L^{\infty}(B_{r/2}(z_0))}+\|\widetilde{v}\|^2_{L^{\infty}(B_{r/2}(z_0))}}\right),
\end{equation}
the doubling index with radius $r$ centered at $z_0$, respectively.
\end{definition}

%{\color{blue}Since in the above section, we have already extended $u$ into $\Omega_R$ analytically, the equation $(\ref{basic equation})$ still holds in $\Omega_R$. So $(\ref{new equations with mu})$ holds in $\Omega_R\times\mathbb{R}$. Moreover, both $\widetilde{u}$ and $\widetilde{v}$ are well defined in $\Omega_R\times\mathbb{R}$.}

%\begin{remark}
%We would like to point out that, although the frequency function defined is not invariant under rescaling, it still works well since the solution $u$ and the boundary $\partial\Omega$ are real analytic.
%\end{remark}

We will show the ``almost monotonicity formula'' for $N(z_0,r)$.

\begin{lemma}\label{monotonicity formula}
For $z_0=(x_0,0)$ with $x_0\in\Omega$, there exist positive constants $C_0$, $C$, and $r_0<R$, such that if $N(z_0,r)\geq C_0$ and $r<r_0$, it holds that
\begin{equation}
\frac{N'(z_0,r)}{N(z_0,r)}\geq-Cr.
\end{equation}
%Here $R$ is the same positive constant as in Theorem $\ref{explain theorem}$.
\end{lemma}

\begin{proof}
Denote
\begin{equation}\label{notation of frequency}
\begin{cases}
D_1(z_0,r)=\int_{B_r(z_0)}|D\widetilde{u}|^2dz;\quad D_2(z_0,r)=\int_{B_r(z_0)}|
D\widetilde{v}|^2dz;
\\D_3(z_0,r)=\int_{B_r(z_0)}\widetilde{u}\widetilde{v}dz;\quad
D_4(z_0,r)=\mu\int_{B_r(z_0)}\widetilde{v}^2dz;\\
H_1(z_0,r)=\int_{\partial B_r(z_0)}\widetilde{u}^2d\sigma,\quad H_2(z_0,r)=\int_{\partial B_r(z_0)}\widetilde{v}^2d\sigma;\\
D(z_0,r)=D_1(z_0,r)+D_2(z_0,r)+D_3(z_0,r)+D_4(z_0,r);\\
H(z_0,r)=H_1(z_0,r)+H_2(z_0,r).
\end{cases}
\end{equation}
Then $$N(z_0,r)=r\frac{D(z_0,r)}{H(z_0,r)}.$$
For $H(z_0,r)$, through the direct calculation, we have
\begin{equation}\label{derivative of H}
\begin{cases}
H_1'(z_0,r)=\frac{n-1}{r}H_1(z_0,r)+\frac{2}{r}\int_{\partial B_r(z_0)}\widetilde{u}\widetilde{u}_{\nu}d\sigma;\\
H_2'(z_0,r)=\frac{n-1}{r}H_2(z_0,r)+\frac{2}{r}\int_{\partial B_r(z_0)}\widetilde{v}\widetilde{v}_{\nu}d\sigma.
\end{cases}
\end{equation}

For $D'(z_0,r)$, we have
\begin{eqnarray}\label{derivative of D1}
D_1'(z_0,r)\nonumber&=&\int_{\partial B_r(z_0)}|D\widetilde{u}|^2d\sigma=\frac{1}{r}\int_{B_r(z_0)}div(|D\widetilde{u}|^2\cdot z)dz
\\\nonumber&=&\frac{n}{r}D_1(z_0,r)
+\frac{2}{r}\int_{B_r(z_0)}\widetilde{u}_i\cdot\widetilde{u}_{ij}\cdot z_jdz
\\\nonumber&=&\frac{n-2}{r}D_1(z_0,r)+\frac{2}{r}\int_{\partial B_r(z_0)}\widetilde{u}^2_{\nu}d\sigma
-\frac{2}{r}\int_{B_r(z_0)}\widetilde{v}D\widetilde{u}\cdot zdz
\\&=&\frac{n-2}{r}D_1(z_0,r)+\frac{2}{r}\int_{\partial B_r(z_0)}\widetilde{u}^2_{\nu}d\sigma-I_1,
\end{eqnarray}
with  $I_1=\frac{2}{r}\int_{B_r(z_0)}\widetilde{v}D\widetilde{u}\cdot zdz$,
\begin{equation}\label{derivative of D2}
D_2'(z_0,r)
=\frac{n-2}{r}D_2(z_0,r)+\frac{2}{r}\int_{\partial B_r(z_0)}\widetilde{v}^2_{\nu}d\sigma
-I_2,
\end{equation}
with
$I_2=\frac{2\mu}{r}\int_{B_r(z_0)}\widetilde{v}D\widetilde{v}\cdot zdz$,
\begin{eqnarray}\label{derivative of D3}
|D_3'(z_0,r)|&=&\left|\int_{\partial B_r(z_0)}\widetilde{u}\widetilde{v}d\sigma\right|
\leq\frac{1}{2}\left(\int_{\partial B_r(z_0)}\widetilde{u}^2d\sigma+\int_{\partial B_r(z_0)}\widetilde{v}^2\sigma\right)
\leq\frac{1}{2}H(z_0,r),
\end{eqnarray}
and
\begin{eqnarray}\label{derivative of D4}
D_4'(z_0,r)\nonumber&=&\mu\int_{\partial B_r(z_0)}\widetilde{v}^2d\sigma
=\frac{\mu}{r}\int_{B_r(z_0)}div(\widetilde{v}^2\cdot z)dz
\\\nonumber&=&\frac{n}{r}D_4(z_0,r)+\frac{2\mu}{r}\int_{B_r(z_0)}\widetilde{v}D\widetilde{v}\cdot zdz
\\&=&\frac{n}{r}D_4(z_0,r)+I_2.
\end{eqnarray}

Now we give an estimate of $\int_{B_r(z_0)}\widetilde{u}^2dz$ and $\int_{B_r(z_0)}\widetilde{v}^2dz$ below.
Let $\widetilde{u}=\widetilde{u}_1+\widetilde{u}_2$ such that $\widetilde{u}_1$ is a harmonic function with $\widetilde{u}_1=\widetilde{u}$ on $\partial B_r(z_0)$. Then, by Corollary $2.2.7$ in $\cite{Han and Lin book}$, we have 
\begin{equation}\label{from Han Lin book}
\int_{B_r(z_0)}\widetilde{u}_1^2dz\leq \frac{r}{n}\int_{\partial B_r(z_0)}\widetilde{u}_1^2d\sigma=\frac{r}{n}\int_{\partial B_r(z_0)}\widetilde{u}^2dz.
\end{equation}
Since $\widetilde{u}_2=\widetilde{u}-\widetilde{u}_1\in W_0^{1,2}(B_r(z_0))$, from the Poincare's inequality,
\begin{equation*}
\int_{B_r(z_0)}\widetilde{u}_2^2dz\leq Cr^2\int_{B_r(z_0)}|D\widetilde{u}_2|^2dz\leq Cr^2\int_{B_r(z_0)}|D\widetilde{u}|^2dz.
\end{equation*}
So
\begin{eqnarray}\label{estimate of u^2}
\int_{B_r(z_0)}\widetilde{u}^2dz&\leq&2\int_{B_r(z_0)}(\widetilde{u}_1^2+\widetilde{u}_2^2)dz\\\nonumber
&\leq& Cr\int_{\partial B_r(z_0)}\widetilde{u}^2d\sigma+Cr^2\int_{B_r(z_0)}|D\widetilde{u}|^2dz.
\end{eqnarray}
By the similar argument to $\widetilde{v}$, we also have
\begin{eqnarray}\label{estimate of v^2}
\int_{B_r(z_0)}\widetilde{v}^2dz&\leq&2\int_{B_r(z_0)}(\widetilde{v}_1^2+\widetilde{v}_2^2)dz\\\nonumber
&\leq& Cr\int_{\partial B_r(z_0)}\widetilde{v}^2d\sigma+Cr^2\int_{B_r(z_0)}|D\widetilde{v}|^2dz.
\end{eqnarray}
Thus
\begin{eqnarray}\label{int u^2 and v^2 by D and H}
\int_{B_r(z_0)}(\widetilde{u}^2+\widetilde{v}^2)dz&\leq& C\left(r\int_{\partial B_r(z_0)}(\widetilde{u}^2+\widetilde{v}^2)d\sigma+r^2\int_{B_r(z_0)}(|D\widetilde{u}|^2+|D\widetilde{v}|^2)dz\right)\nonumber
\\&=&Cr^2(D_1(z_0,r)+D_2(z_0,r))+CrH(z_0,r).
\end{eqnarray}

From $(\ref{derivative of D1})$ and $(\ref{int u^2 and v^2 by D and H})$, we have
\begin{eqnarray}\label{estimate of I1}
|I_1|\nonumber&=&\frac{2}{r}\left|\int_{B_r(z_0)}\widetilde{v}D\widetilde{u}\cdot zdz\right|
\leq2\int_{B_r(z_0)}|\widetilde{v}||D\widetilde{u}|dz
\\\nonumber&\leq&\frac{1}{r}\int_{B_r(z_0)}\widetilde{v}^2dz+r\int_{B_r(z_0)}|D\widetilde{u}|^2dz
\\&\leq&Cr\left(D_1(z_0,r)+D_2(z_0,r)\right)+CH(z_0,r).
\end{eqnarray}
So from $(\ref{derivative of D1}-\ref{estimate of I1})$, there holds
\begin{eqnarray}\label{D'}
D'(z_0,r)&=&\frac{n-2}{r}\left(D_1(z_0,r)+D_2(z_0,r)+D_4(z_0,r)\right)+D_3'(z_0,r)\\\nonumber&+&\frac{2}{r}D_4(z_0,r)+\frac{2}{r}\left(\int_{\partial B_r(z_0)}\widetilde{u}\widetilde{u}_{\nu}d\sigma+\int_{\partial B_r(z_0)}\widetilde{v}\widetilde{v}_{\nu}d\sigma\right)-I_1
\\\nonumber&\geq&\frac{n-2}{r}D(z_0,r)+\frac{2}{r}\int_{\partial B_r(z_0)}\left(\widetilde{u}_{\nu}^2+\widetilde{v}_{\nu}^2\right)d\sigma\\\nonumber&-&Cr(D_1(z_0,r)+D_2(z_0,r))-CH(z_0,r)-\frac{n-2}{r}|D_3(z_0,r)|.
\end{eqnarray}

Next, we need to estimate the upper bound for the term $|D_3(z_0,r)|$. In fact, from $(\ref{estimate of u^2})$ and $(\ref{estimate of v^2})$,
\begin{eqnarray}\label{form of D3}
|D_3(z_0,r)|&=&\nonumber\left|\int_{B_r(z_0)}\widetilde{u}\widetilde{v}dz\right|
\\\nonumber&\leq&\frac{1}{2}\left(\int_{B_r(z_0)}\widetilde{u}^2+\int_{B_r(z_0)}\widetilde{v}^2dz\right)
\\&\leq&Cr^2(D_1(z_0,r)+D_2(z_0,r))+CrH(z_0,r).
\end{eqnarray}
For $N(z_0,r)\geq C_0$, 
\begin{equation}\label{estimate of H by D}
H(z_0,r)\leq \frac{r}{C_0}D(z_0,r).
\end{equation}
Thus from $(\ref{form of D3})$ and $(\ref{estimate of H by D})$,
\begin{equation*}
|D_3(z_0,r)|\leq Cr^2(D_1(z_0,r)+D_2(z_0,r))+\frac{C}{C_0}r^2D(z_0,r),
\end{equation*}
provided that $N(z_0,r)\geq C_0$. Then for any $r\leq r_0$ with $r_0$ small enough such that $\frac{C}{C_0}r^2<\frac{1}{2}$, 
\begin{equation*}
|D_3(z_0,r)|\leq Cr^2(D(z_0,r)+|D_3(z_0,r)|)\leq Cr^2D(z_0,r)+\frac{1}{2}|D_3(z_0,r)|,
\end{equation*}
which implies that
\begin{equation}\label{estimate of D3 by D}
|D_3(z_0,r)|\leq Cr^2D(z_0,r).
\end{equation}
By putting $(\ref{estimate of D3 by D})$ into $(\ref{D'})$, we have
\begin{equation}
\frac{D'(z_0,r)}{D(z_0,r)}\geq \frac{n-2}{r}+\frac{2}{r}\frac{\int_{\partial B_r(z_0)}(\widetilde{u}^2_{\nu}+\widetilde{v}^2_{\nu})d\sigma}{\int_{\partial B_r(z_0)}(\widetilde{u}\widetilde{u}_{\nu}+\widetilde{v}\widetilde{v}_{\nu})d\sigma}-Cr.
\end{equation}
From the Cauchy inequality, there holds
\begin{eqnarray}
\int_{\partial B_r(z_0)}(\widetilde{u}\widetilde{u}_{\nu}+\widetilde{v}\widetilde{v}_{\nu})d\sigma
\leq\left(\int_{\partial B_r(z_0)}(\widetilde{u}^2+\widetilde{v}^2)d\sigma\right)^{\frac{1}{2}}\left(\int_{\partial B_r(z_0)}(\widetilde{u}_{\nu}^2+\widetilde{v}^2_{\nu})d\sigma\right)^{\frac{1}{2}}.
\end{eqnarray}
So
\begin{equation}
\left(\frac{\int_{\partial B_r(z_0)}(\widetilde{u}^2_{\nu}+\widetilde{v}^2_{\nu})d\sigma}{\int_{\partial B_r(z_0)}(\widetilde{u}\widetilde{u}_{\nu}+\widetilde{v}\widetilde{v}_{\nu})d\sigma}
-\frac{\int_{\partial B_r(z_0)}(\widetilde{u}\widetilde{u}_{\nu}+\widetilde{v}\widetilde{v}_{\nu})d\sigma}{\int_{\partial B_r(z_0)}(\widetilde{u}^2+\widetilde{v}^2)d\sigma}\right)
\geq0.
\end{equation}
Then from the derivative of $H(z_0,r)$, and the direct calculation of $N'(z_0,r)$,
\begin{eqnarray}
\frac{N'(z_0,r)}{N(z_0,r)}&=& \frac{1}{r}+\frac{D'(z_0,r)}{D(z_0,r)}-\frac{H'(z_0,r)}{H(z_0,r)}\nonumber
\\&\geq&\frac{2}{r}\left(\frac{\int_{\partial B_r(z_0)}(\widetilde{u}^2_{\nu}+\widetilde{v}^2_{\nu})d\sigma}{\int_{\partial B_r(z_0)}(\widetilde{u}\widetilde{u}_{\nu}+\widetilde{v}\widetilde{v}_{\nu})d\sigma}
-\frac{\int_{\partial B_r(z_0)}(\widetilde{u}\widetilde{u}_{\nu}+\widetilde{v}\widetilde{v}_{\nu})d\sigma}{\int_{\partial B_r(z_0)}(\widetilde{u}^2+\widetilde{v}^2)d\sigma}\right)-Cr\nonumber
\\&\geq&-Cr,
\end{eqnarray}
which is the desired result. 
\end{proof}

Such frequency functions also have a lower bound as follows.

\begin{lemma}
There exists positive constant $r_0'$ depending only on $n$ and $\Omega$, such that if $r\leq r_0'$, then
\begin{equation}
N(z_0,r)\geq -Cr^2.
\end{equation}
Here $C$ is a positive constant depending only on $n$ and $\Omega$.
\end{lemma}

\begin{proof}
From $(\ref{notation of frequency})$, we only need to estimate $D_3(z_0,r)$, since other terms are all positive. From the H\"older inequality and the inequalities $(\ref{estimate of u^2})$ and $\ref{estimate of v^2})$.

\begin{eqnarray}
|D_3(z_0,r)|&\leq&\left(\int_{B_r(z_0)}\widetilde{u}^2dz\right)^{1/2}\left(\int_{B_r(z_0)}\widetilde{v}^2dz\right)^{1/2}\nonumber
\\&\leq&\frac{1}{2}\left(\int_{B_r(z_0)}\widetilde{u}^2dz+\int_{B_r(z_0)}\widetilde{v}^2dz\right)\nonumber
\\&\leq&Cr\left(\int_{\partial B_r(z_0)}\widetilde{u}^2d\sigma+\int_{\partial B_r(z_0)}\widetilde{v}^2d\sigma\right)\nonumber
+Cr^2\left(\int_{B_r(z_0)}|D\widetilde{u}|^2dz+\int_{B_r(z_0)}|D\widetilde{v}|^2dz\right)\nonumber
\\&=&Cr\left(H_1(z_0,r)+H_2(z_0,r)\right)+Cr^2\left(D_1(z_0,r)+D_2(z_0,r)\right).
\end{eqnarray}
Here $C$ is a positive constant depending only on $n$. Thus
\begin{eqnarray}
N(z_0,r)&\geq&r\frac{D_1(z_0,r)+D_2(z_0,r)-|D_3(z_0,r)|+D_4(z_0,r)}{H_1(z_0,r)+H_2(z_0,r)}\nonumber
\\&\geq&r\frac{(1-Cr^2)(D_1(z_0,r)+D_2(z_0,r))-Cr(H_1(z_0,r)+H_2(z_0,r))}{H_1(z_0,r)+H_2(z_0,r)}\nonumber
\\&\geq&-Cr^2,
\end{eqnarray}
provided that $r>0$ is small enough such that $1-Cr^2\geq0$.
This completes the proof.
\end{proof}

%{\color{red}From this lemma, $N(z_0,r)+1>0$ provided that $r>0$ is small enough such that $Cr^2<1$, where $C$ is the same positive constant as in Lemma $\ref{lower bound}$.}
We can get the following doubling conditions from Lemma $\ref{monotonicity formula}$ and the ``almost monotonicity formula''.
\begin{lemma}\label{doubling condition}
Let $r_0$ be the same positive constant as in Lemma $\ref{monotonicity formula}$. For $z_0=(x_0,0)$ and $r<r_0$, it holds that
\begin{equation}
\begin{cases}
\fint_{\partial B_r(z_0)}(\widetilde{u}^2+\widetilde{v}^2)d\sigma\leq2^{C(N(z_0,r)+1)}\fint_{\partial B_{r/2}(z_0)}(\widetilde{u}^2+\widetilde{v}^2)d\sigma,\\
\fint_{ B_r(z_0)}(\widetilde{u}^2+\widetilde{v}^2)dz\leq2^{C(N(z_0,r)+1)}\fint_{ B_{r/2}(z_0)}(\widetilde{u}^2+\widetilde{v}^2)dz,\\
\fint_{\partial B_r(z_0)}(\widetilde{u}^2+\widetilde{v}^2)d\sigma\geq2^{CN(z_0,r/2)-C'}\fint_{\partial B_{r/2}(z_0)}(\widetilde{u}^2+\widetilde{v}^2)d\sigma,\\
\fint_{ B_r(z_0)}(\widetilde{u}^2+\widetilde{v}^2)dz\geq2^{CN(z_0,r)-C'}\fint_{ B_{r/2}(z_0)}(\widetilde{u}^2+\widetilde{v}^2)dz,
\end{cases}
\end{equation}
where $C$ and $C'$ in different forms are different positive constants depending only on $n$. 
\end{lemma}

\begin{proof}
This is a direct result by taking integration on the quantity $\frac{N'(z_0,r)}{N(z_0,r)}$. From the calculation of $H'(z_0,r)$ in the proof of Lemma $\ref{monotonicity formula}$,
\begin{equation}\label{derivative of lnH}
\frac{d}{dr}\left(\log H(z_0,r)\right)=\frac{H'(z_0,r)}{H(z_0,r)}=\frac{n-1}{r}+2\frac{\int_{\partial B_r(z_0)}(\widetilde{u}\widetilde{u}_{\nu}+\widetilde{v}\widetilde{v}_{\nu})d\sigma}{\int_{\partial B_r(z_0)}(\widetilde{u}^2+\widetilde{v}^2)d\sigma}
=\frac{n-1}{r}+2\frac{N(z_0,r)}{r}.
\end{equation}
Thus
\begin{equation}
\ln\frac{H(z_0,r)}{H\left(z_0,\frac{r}{2}\right)}=\int_{\frac{r}{2}}^r\frac{H'(z_0,\rho)}{H(z_0,\rho)}d\rho
=\int_{\frac{r}{2}}^r\frac{n-1+2N(z_0,\rho)}{\rho}d\rho.
\end{equation}
From the monotonicity formula, we know that for any $\rho<r$,
\begin{equation}\label{same center control}
N(z_0,\rho)\leq C(N(z_0,r)+1),
\end{equation}
for some $C>0$ depending only on the dimension $n$. Then
\begin{equation*}
\ln\frac{H(z_0,r)}{H\left(z_0,\frac{r}{2}\right)}\leq C(N(z_0,r)+1),
\end{equation*}
and then
\begin{equation}
H(z_0,r)\leq 2^{C(N(z_0,r)+1)}H\left(z_0,\frac{r}{2}\right),
\end{equation}
where $C$ is a positive constant depending only on $n$. This is the first inequality of this lemma. The second inequality of $(\ref{doubling condition})$ can be obtained by the first one. 

Now we prove the third and the fourth inequalities. In fact, from the monotonicity formula, for any $\rho\in(r/2,r)$, we have
\begin{equation*}
N(z_0,\rho)\geq CN(z_0,r/2)-C'.
\end{equation*}
Thus
\begin{eqnarray*}
\ln\frac{H(z_0,r)}{H\left(z_0,\frac{r}{2}\right)}&=&\int_{\frac{r}{2}}^r\frac{n-1+2N(z_0,\rho)}{\rho}d\rho\geq\int_{\frac{r}{2}}^r\frac{n-1+CN(z_0,r/2)-C'}{\rho}d\rho
\\&\geq&CN(z_0,r/2)-C'.
\end{eqnarray*}
Then
\begin{equation}
H(z_0,r)\geq 2^{CN(z_0,r/2)-C'}H\left(z_0,\frac{r}{2}\right),
\end{equation}
where $C$ and $C'$ are positive constants depending only on $n$. This is the third inequality. The fourth one can be derived by integrating the third one.
\end{proof}

\begin{remark}\label{doubling condition of ball}
By the similar arguments as in the proof of Lemma $\ref{doubling condition}$, we have for $0<r_1<r_2\leq r_0$,
\begin{equation}
\begin{cases}
\fint_{B_{r_2}(z_0)}(\widetilde{u}^2+\widetilde{v}^2)dz\leq\left(\frac{r_2}{r_1}\right)^{C(N(z_0,r)+1)}\fint_{B_{r_1}(z_0)}(\widetilde{u}^2+\widetilde{v}^2)dz,\\
\fint_{B_{r_2}(z_0)}(\widetilde{u}^2+\widetilde{v}^2)dz\geq\left(\frac{r_2}{r_1}\right)^{CN(z_0,r/2)-C'}\fint_{B_{r_1}(z_0)}(\widetilde{u}^2+\widetilde{v}^2)dz.
\end{cases}
\end{equation}
\end{remark}

%On the other hand, the frequency function is also bounded from below.
%\begin{lemma}\label{lower bound}
%If $r\leq r_0$, then
%\begin{equation}
%N(z_0,r)\geq -Cr^2,
%\end{equation}
%where $C$ and $r_0$ are positive constants depending only on $n$.
%\end{lemma}

%From the calculation of $D_3(z_0,r)$ in the proof of Lemma $\ref{monotonicity formula}$, we have
%\begin{equation*}
%|D_3(z_0,r)|\leq Cr^2(D_1(z_0,r)+D_2(z_0,r))+CrH(z_0,r),
%\end{equation*}
%for some positive constant $C$ depending only on $n$. Then for $r>0$ small enough, and noting that $D_4(z_0,r)\geq0,$
%\begin{eqnarray*}
%N(z_0,r)&\geq&r\frac{D_1(z_0,r)+D_2(z_0,r)-|D_3(z_0,r)|+D_4(z_0,r)}{H(z_0,r)}
%\\&\geq&r\frac{-CrH(z_0,r)}{H(z_0,r)}\geq -Cr^2.
%\end{eqnarray*}
%\end{proof}

Now we can establish the  ``changing center property''.
\begin{lemma}\label{changing center}
Let $z_1\in B_{r/4}(z_0)$ with $z_1=(x_1,0)$ and $x_1\in\Omega$. Then for $\rho\leq r/4$, we have
\begin{equation}
N(z_1,\rho)\leq C(N(z_0,r)+1),
\end{equation}
 where $C$ is a positive constant depending only on $n$.
\end{lemma}

\begin{proof}
From $(\ref{derivative of lnH})$, for  $\rho=\frac{r}{4}$ and any $t\in\left(\frac{3\rho}{2},2\rho\right)$, we have
\begin{equation}\label{lower control of H}
\ln\frac{\fint_{\partial B_t(z_1)}(\widetilde{u}^2+\widetilde{v}^2)d\sigma}{\fint_{\partial B_{3\rho/2}(z_1)}(\widetilde{u}^2+\widetilde{v}^2)d\sigma}
=\int_{3\rho/2}^{t}\frac{2N(z_1,l)}{l}dl
\geq-C\ln\frac{2t}{3\rho}\geq-C,
\end{equation}
which implies that
\begin{eqnarray*}
\fint_{\partial B_{3\rho/2}(z_1)}(\widetilde{u}^2+\widetilde{v}^2)d\sigma\leq
C\fint_{\partial B_t(z_1)}(\widetilde{u}^2+\widetilde{v}^2)d\sigma.
\end{eqnarray*}
for any $t\in \left(\frac{3\rho}{2},2\rho\right)$. Then
\begin{equation}\label{upper bound of fenzi}
\fint_{\partial B_{3\rho/2}(z_1)}(\widetilde{u}^2+\widetilde{v}^2)d\sigma\leq
C\fint_{B_{2\rho}(z_1)\setminus B_{3\rho/2}(z_1)}(\widetilde{u}^2+\widetilde{v}^2)dz
\leq C\fint_{B_r(z_0)}(\widetilde{u}^2+\widetilde{v}^2)dz.
\end{equation}
Here we have used the fact that $B_{2\rho}(z_1)\subseteq B_r(z_0)$. By the similar argument as in the proof of $(\ref{lower control of H})$, we have
\begin{equation*}
\ln\frac{\fint_{\partial B_{5\rho/4}(z_1)}\left(\widetilde{u}^2+\widetilde{v}^2\right)d\sigma}{\fint_{\partial B_{t}(z_1)}\left(\widetilde{u}^2+\widetilde{v}^2\right)d\sigma}
\geq-C,
\end{equation*}
for any $t\in\left(0,\frac{5\rho}{4}\right)$. Then because $B_{\rho/4}(z_0)\subseteq B_{5\rho/4}(z_1)$,
\begin{eqnarray}\label{lower bound of fenmu}
\fint_{\partial B_{5\rho/4}(z_1)}\left(\widetilde{u}^2+\widetilde{v}^2\right)d\sigma
\geq\frac{1}{C}\fint_{B_{5\rho/4}(z_1)}\left(\widetilde{u}^2+\widetilde{v}^2\right)dz\geq\frac{1}{C}\fint_{B_{\rho/4}(z_0)}\left(\widetilde{u}^2+\widetilde{v}^2\right)dz.
\end{eqnarray}
From Lemma $\ref{doubling condition}$, we also have
\begin{eqnarray}\label{doubling of z0}
\fint_{B_{r}(z_0)}\left(\widetilde{u}^2+\widetilde{v}^2\right)dz\leq 2^{C(N(z_0,r)+1)}\fint_{B_{\rho/4}(z_0)}\left(\widetilde{u}^2+\widetilde{v}^2\right)dz.
\end{eqnarray}
So from $(\ref{upper bound of fenzi})$, $(\ref{lower bound of fenmu})$, and $(\ref{doubling of z0})$,
\begin{eqnarray*}
N(z_1,\rho)\leq C\ln\frac{\fint_{\partial B_{3\rho/2}(z_1)}(\widetilde{u}^2+\widetilde{v}^2)d\sigma}{\fint_{\partial B_{5\rho/4}(z_1)}\left(\widetilde{u}^2+\widetilde{v}^2\right)d\sigma}\leq C\ln\frac{C\fint_{B_r(z_0)}(\widetilde{u}^2+\widetilde{v}^2)dz}{C^{-1}\fint_{B_{\rho/4}(z_0)}\left(\widetilde{u}^2+\widetilde{v}^2\right)dz}
\leq C(N(z_0,r)+1).
\end{eqnarray*}
\end{proof}

From the above lemmas and Sobolev's Embedding Theorem, we can derive the relationship between the frequency function and the doubling index.
\begin{lemma}\label{relationship}
If $\mu>0$ large enough, there exist positive constants $C$, $c$, $\widetilde{C}$ and $\widetilde{c}$ depending only on $n$, such that for any $\eta\in(0,1/2)$,
\begin{equation}\label{N by M}
N(z_0,r)\leq cM(z_0,(\eta+1) r)+\widetilde{c}(1-\log_2\eta-\log_2r),
\end{equation}
and
\begin{equation}\label{M by N}
M(z_0,r)\leq CN(z_0,(\eta+1) r)+\widetilde{C}(1-\log_2\eta-\log_2r).
\end{equation}
with $z_0=(x_0,0)$.
\end{lemma}

\begin{proof}
First we will give interior estimates of $\widetilde{u}$ and $\widetilde{v}$. Let $B_r(z_0)\subseteq\Omega_R\times\mathbb{R}$ be a fixed ball. Let $\phi$ be the cut-off function of $B_r(z_0)$ such that $\phi=1$ in $B_{(1-\eta) r}(z_0)$, $\phi=0$ outside $B_{r}(z_0)$, and $|D\phi|\leq\frac{C}{\eta r}$. Then by multiplying $\widetilde{u}\phi^2$ on both sides of the first equation in $(\ref{new equations with mu})$, and taking integration by parts, we have
\begin{eqnarray*}
\int_{B_r(z_0)}|D\widetilde{u}|^2\phi^2dz&=&-2\int_{B_r(z_0)}\widetilde{u}\phi D\widetilde{u}D\phi dz-\int_{B_r(z_0)}\widetilde{u}\widetilde{v}\phi^2dz
\\&\leq&\frac{1}{2}\int_{B_r(z_0)}|D\widetilde{u}|^2\phi^2dz+2\int_{B_r(z_0)}\widetilde{u}^2|D\phi|^2dz+\frac{1}{2}\left(\int_{B_r(z_0)}\widetilde{u}^2\phi^2dz+\int_{B_r(z_0)}\widetilde{v}^2\phi^2dz\right).
\end{eqnarray*}
This implies that
\begin{equation}
\|\widetilde{u}\|_{W^{1,2}(B_{(1-\eta) r}(z_0))}\leq C\left((\eta r)^{-1}\|\widetilde{u}\|_{L^2(B_r(z_0))}+\|\widetilde{v}\|_{L^2(B_r(z_0))}\right).
\end{equation}
Similarly, by multiplying $\widetilde{v}\phi^2$ on both sides of the second equation in $(\ref{new equations with mu})$, we have
\begin{eqnarray}
\int_{B_r(z_0)}|D\widetilde{v}|^2\phi^2dz&=&-2\int_{B_r(z_0)}\widetilde{v}\phi D\widetilde{v}D\phi dz-\mu\int_{B_r(z_0)}\widetilde{v}^2\phi^2dz\nonumber
\\&\leq&\frac{1}{2}\int_{B_r(z_0)}|D\widetilde{v}|^2\phi^2dz+2\int_{B_r(z_0)}\widetilde{v}^2|D\phi|^2dz.
\end{eqnarray}
This implies that
\begin{equation*}
\|\widetilde{v}\|_{W^{1,2}(B_{(1-\eta) r}(z_0))}\leq \frac{C}{\eta r}\|\widetilde{v}\|_{L^{2}(B_r(z_0))},
\end{equation*}
and then
\begin{equation*}
\|\widetilde{v}\|_{W^{k,2}(B_{(1-\eta) r}(z_0))}\leq \frac{C}{\eta r}\|\widetilde{v}\|_{W^{k-1,2}(B_r(z_0))},
\end{equation*}
So by the iteration argument and Sobolev's Embedding Theorem, for any $B_{r}(z_0)\subseteq\Omega_R\times\mathbb{R}$,
\begin{equation}
\|\widetilde{u}\|_{L^{\infty}(B_{(1-\eta)r}(z_0))}+\|\widetilde{v}\|_{L^{\infty}(B_{(1-\eta) r}(z_0))}\leq \frac{C}{(\eta r)^{\frac{n+2}{2}}}\left(\|\widetilde{u}\|_{L^2(B_{r}(z_0))}+\|\widetilde{v}\|_{L^2(B_{r}(z_0))}\right).
\end{equation}
Thus from Lemma $\ref{doubling condition}$ and Remark $\ref{doubling condition of ball}$, we have
\begin{eqnarray}
M(z_0,r)&=&\frac{1}{2}\log_2\frac{\|\widetilde{u}\|^2_{L^{\infty}(B_r(z_0))}+\|\widetilde{v}\|^2_{L^{\infty}(B_r(z_0))}}{\|\widetilde{u}\|^2_{L^{\infty}(B_{r/2}(z_0))}+\|\widetilde{v}\|^2_{L^{\infty}(B_{r/2}(z_0))}}\nonumber
\\&\leq&C\left(-\log_2\eta-\log_2 r\right)
+\frac{1}{2}\log_2\frac{\|\widetilde{u}\|^2_{L^2(B_{(\eta+1) r}(z_0))}+\|\widetilde{v}\|^2_{L^2(B_{(\eta+1) r}(z_0))}}{\|\widetilde{u}\|^2_{L^2(B_{r/2}(z_0))}+\|\widetilde{v}\|^2_{L^2(B_{r/2}(z_0))}}\nonumber
\\&\leq&C\left(-\log_2\eta-\log_2 r\right)+C(N(z_0,(\eta+1) r)+1).
\end{eqnarray}
which is the inequality $(\ref{M by N})$.

Inequality $(\ref{N by M})$ can be obtained by similar arguments. In fact, from Lemma $\ref{doubling condition}$ again, we have
\begin{eqnarray}
M(z_0,(1+\eta)r)&=&\frac{1}{2}\log_2\frac{\|\widetilde{u}\|^2_{L^{\infty}(B_{(1+\eta)r}(z_0))}+\|\widetilde{v}\|^2_{L^{\infty}(B_{(1+\eta)r}(z_0))}}{\|\widetilde{u}\|^2_{L^{\infty}(B_{\frac{(1+\eta)r}{2}}(z_0))}+\|\widetilde{v}\|^2_{L^{\infty}(B_{\frac{(1+\eta)r}{2}}(z_0))}}\nonumber
\\&\geq&-C\left(-\log2\eta-\log_2r\right)+\frac{1}{2}\log_2\frac{\|\widetilde{u}\|^2_{L^2(B_{(1+\eta)r}(z_0))}+\|\widetilde{v}\|^2_{L^2(B_{(1+\eta)r}(z_0))}}{\|\widetilde{u}\|^2_{L^2(B_{r/2}(z_0))}+\|\widetilde{v}\|^2_{L^2(B_{r/2}(z_0))}}\nonumber
\\&\geq&C\left(-\log_2\eta-\log_2 r\right)+CN(z_0,r/2)-C'.
\end{eqnarray}
Then the first inequality of this Lemma is obtained.
\end{proof}

Now we are ready to give an upper bound for the frequency function and the doubling index.

\begin{lemma}\label{upper bound of frequency}
There exist positive constants $C$ and $R_0$ depending only on $n$ and $\Omega$, such that for any $z_0=(x_0,0)$ with $x_0\in\Omega$ and $r\leq R_0/2$, it holds that
\begin{equation}
N(z_0,r)\leq C\sqrt{\mu},
\end{equation}
provided that $B_{r}(x_0)\subseteq\Omega_R$ and $\mu>0$ large enough.
\end{lemma}

\begin{proof}
Without loss of generality, assume that $\|u\|_{L^2(\Omega)}=1$. Then from Lemma $\ref{explain theorem}$ and the relationship between $u$ and $\widetilde{u}$, we have

\begin{eqnarray}
\|\widetilde{u}\|_{L^{\infty}(\Omega\times(-R,R))}\leq e^{C\sqrt{\mu}R}\|u\|_{L^{\infty}(\Omega)}\leq  e^{C(\sqrt{\mu}+\Lambda^{1/4})R}\|u\|_{L^2(\Omega)},
\end{eqnarray}
where $R$ is the same positive constant as in Lemma $\ref{explain theorem}$. From the proof of Lemma $\ref{explain theorem}$ and the relationship between $u$ and $\widetilde{u}$ again,
\begin{equation}
\|\widetilde{v}\|_{L^{\infty}(\Omega\times(-R,R))}\leq C\|\widetilde{u}\|_{W^{2,\infty}(\Omega\times(-R,R))}\leq e^{C\sqrt{\mu}R}\|u\|_{W^{2,\infty}(\Omega)}
\leq e^{C(\sqrt{\mu}+\Lambda^{1/4})R}\|u\|_{L^2(\Omega)},
\end{equation}
where $C$ in different terms are different positive constants depending only on $n$ and $\Omega$. So
\begin{equation}\label{upper bound of Linfty of u and v}
\|\widetilde{u}\|^2_{L^{\infty}(\Omega\times(-R,R))}+\|\widetilde{v}\|^2_{L^{\infty}(\Omega\times(-R,R))}\leq e^{C(\sqrt{\mu}+\Lambda^{1/4})R}.
\end{equation}
Let $\bar{x}$ be the maximum point of $u$ in $\overline{\Omega}$ and $\bar{z}=(\bar{x},0)$. Since $\|u\|_{L^2(\Omega)}=1$, there holds
\begin{equation}
|u(\bar{x})|=\|u\|_{L^{\infty}(\Omega)}\geq\frac{\|u\|_{L^2(\Omega)}}{\sqrt{|\Omega|}}=|\Omega|^{-\frac{1}{2}}.
\end{equation}
Here $|\Omega|$ means the $n$ dimensional Hausdorff measure of $\Omega$.
Then for any $r<R$, from $(\ref{upper bound of Linfty of u and v})$, 
\begin{equation}
M(\bar{z},r)=\frac{1}{2}\log_2\frac{\|\widetilde{u}\|^2_{L^{\infty}(B_r(\bar{z}))}+\|\widetilde{v}\|^2_{L^{\infty}(B_r(\bar{z}))}}{\|\widetilde{u}\|^2_{L^{\infty}(B_r(\bar{z}))}+\|\widetilde{v}\|^2_{L^{\infty}(B_r(\bar{z}))}}\leq \frac{1}{2}\log_2\frac{e^{C(\sqrt{\mu}+\Lambda^{1/4})R}}{u(\bar{x})}\leq C(\sqrt{\mu}+\Lambda^{1/4}),
\end{equation}
where $C$ is a positive constant depending on $n$, $\Omega$ and $R$. In the first inequality above we have also used the assumption that $\mu>0$ is large enough. Then by Lemma $\ref{relationship}$ with $\eta=\frac{1}{4}$, and noting that $\widetilde{u}(x,x_{n+1})=u(x)e^{\sqrt{\frac{\lambda+\mu}{2}}x_{n+1}}$ and $\widetilde{v}(x,x_{n+1})=\left(\triangle u(x)+\frac{\lambda+\mu}{2}u(x)\right)e^{\sqrt{\frac{\lambda+\mu}{2}}x_{n+1}}$, we have for $r\leq R_0$, with $R_0=\min\left\{r_0,R/4\right\}$, such that $N(\bar{z},R_0)\leq C\sqrt{\mu}$ with $\bar{z}=(\bar{x},0)$, provided that $\mu>0$ large enough. Then from Lemma $\ref{changing center}$ and Lemma $\ref{relationship}$,
$N\left(z,\frac{R_0}{4}\right), M\left(z,\frac{R_0}{4}\right)\leq C(\sqrt{\mu}+\Lambda^{1/4})$, where $z\in B_{\frac{R_0}{4}}(\bar{z})$ with $z=(x,0)$ and $x\in\Omega$. So $\|\widetilde{u}\|_{L^{\infty}(B_{\frac{R_0}{2}}(z))}\geq e^{-C(\sqrt{\mu}+\Lambda^{1/4})}$. This implies that $M(z,R_0)\leq C\sqrt{\mu}$ for above $z$. By the similar argument for finitely many steps, where the number of the steps depends only on $\Omega$, $R$ and $R_0$, we have that for any $z=(x,0)$ with $x\in\Omega$,  $M(z, R_0)\leq C(\sqrt{\mu}+\Lambda^{1/4}).$ Then by the fact that $\Lambda=\left(\frac{\mu}{2}\right)^2$ and Lemma $\ref{relationship}$ again, it holds that $N(z, 2R_0/3)\leq C(\sqrt{\mu}+\Lambda^{1/4})\leq C\sqrt{\mu}.$ By the inequality $(\ref{same center control})$,
\begin{equation}
N(z,r)\leq C\sqrt{\mu},\quad \forall\ r\leq \frac{R_0}{2}.
\end{equation}
This completes the proof. 
\end{proof}

Now we arrive at proving Theorem $\ref{quantitative uniqueness}$.

\textbf{Proof of Theorem $\ref{quantitative uniqueness}$:}

Without loss of generality, assume that $z_0=(0,0)$. Let $m$ and $l$ be the vanishing order of $\widetilde{u}$ and $\widetilde{v}=\triangle\widetilde{u}$ at the origin $(0,0)$, respectively. Recall the definition of the vanishing order, we have that
\begin{equation}
\begin{cases}
D^{\alpha}\widetilde{u}(0)=0,\ \ for\ any\ \ |\alpha|<k,\quad D^{\alpha}\widetilde{u}(0)\neq0\ \ for\ some\ \ |\alpha|=m;\\
D^{\alpha}\widetilde{v}(0)=0,\ \ for \ any\ \ |\alpha|<l,\quad D^{\alpha}\widetilde{v}(0)\neq0\ \ for\ some\ \ |\alpha|=l.
\end{cases}
\end{equation}
Thus for $r>0$ small enough, we can rewrite $\widetilde{u}$ and $\widetilde{v}$ as follows.
\begin{equation}\label{Taylor of u and v}
\begin{cases}
\widetilde{u}(z)=r^m\phi(\theta)+o(r^m),\\
\widetilde{v}(z)=r^l\psi(\theta)+o(r^l).
\end{cases}
\end{equation}
Here $r=|z|$, $(r,\theta)$ is the spherical coordinates of $z$,  $\phi$ and $\psi$ are analytic functions of $\theta$.
Now we claim that
\begin{equation}\label{estimate of k and l}
\lim\limits_{r\rightarrow0+}N(0,r)= \min\left\{m,l\right\}.
\end{equation}
In fact,
\begin{eqnarray}
\lim\limits_{r\rightarrow0+}N(0,r)&=&\lim\limits_{r\rightarrow0+}r\frac{\int_{\partial B_r(0)}(\widetilde{u}\widetilde{u}_{\nu}+\widetilde{v}\widetilde{v}_{\nu})d\sigma}{\int_{\partial B_r(0)}(\widetilde{u}^2+\widetilde{v}^2)d\sigma}\nonumber
\\&=&\lim\limits_{r\rightarrow0+}\frac{\int_{\partial B_r(0)}(mr^{2m}\phi^2(\theta)+lr^{2l}\psi^2(\theta)+o(r^{2m}+o(r^{2l})))d\sigma}{\int_{\partial B_r(0)}(r^{2m}\phi^2(\theta)+r^{2l}\psi(\theta)+o(r^{2m}+o(r^{2l}))d\sigma}\nonumber
\\&=&\min\left\{m,l\right\}.
\end{eqnarray}
From Lemma $\ref{upper bound of frequency}$, we have $\min\left\{m,l\right\}\leq C\sqrt{\mu}$. This means that the vanishing order of $\widetilde{u}$ is less than or equal to $C\sqrt{\mu}$, since it is observed that $m\leq l+2$. Then from the relationship of $\widetilde{u}$ and $u$, i.e., $\widetilde{u}(x,x_{n+1})=u(x)e^{\sqrt{\frac{\lambda+\mu}{2}}x_{n+1}}$ with $\mu=\sqrt{\lambda^2+4k^2}$, the conclusion of Theorem $\ref{quantitative uniqueness}$ is obtained.
\qed

%%%%%%%%%%%%%%%%%%%%%%%%%%%%%%%%%%%%%%%%%%%%%%%%%%%

\section{Measure estimate for the nodal set}

The doubling estimates in the above section are established for $\left(\|\widetilde{u}\|_{L^2}^2+\|\widetilde{v}\|_{L^2}^2\right)$. We will give below a new doubling estimate for $\|\widetilde{u}\|_{L^2}^2$.
\begin{lemma}\label{new doubling condition}
There exist positive constants $\bar{r}$, $C_1$, and $C_2$ depending onlly on $n$, such that for any $r\leq \bar{r}/2$, $\eta\in\left(0,\frac{1}{3}\right)$, and $x_0\in\Omega$ with $B_{r}(x_0)\subseteq \Omega_R$, 
\begin{equation}
\int_{B_{(1+\eta)r}(z_0)}\widetilde{u}^2dz\leq C_2(1+3\eta)^{C_1\sqrt{\mu}}\left(\mu^2+\eta^{-4}r^{-4}\right)\int_{B_{r}(z_0)}\widetilde{u}^2dz,
\end{equation}
where $z_0=(x_0,0)$.
\end{lemma}

\begin{proof}

From Lemma $\ref{doubling condition}$ and Lemma $\ref{upper bound of frequency}$,
\begin{eqnarray}\label{old doubling condition}
\int_{B_{(1+\eta)r}(z_0)}\widetilde{u}^2dz\leq\int_{B_{(1+\eta)r}(z_0)}\left(\widetilde{u}^2+\widetilde{v}^2\right)dz
\leq\left(\frac{1+\eta}{1-\eta}\right)^{C_1\sqrt{\mu}}\int_{B_{(1-\eta)r}(z_0)}\left(\widetilde{u}^2+\widetilde{v}^2\right)dz.
\end{eqnarray}
By the same argument as in the proof of Lemma $\ref{interior estimate}$, it holds that 
\begin{eqnarray}
\|\widetilde{v}\|_{L^2(B_{(1-\eta)r}(z_0))}\leq C(\mu+r^{-2}\eta^{-2})\|\widetilde{u}\|_{L^2(B_r(z_0))}.
\end{eqnarray}
%by Lemma $\ref{control v by u}$, 
Then we have
\begin{eqnarray}
\int_{B_{(1+\eta)r}(z_0)}\widetilde{u}^2dz &\leq&\left(\frac{1+\eta}{1-\eta}\right)^{C_1\sqrt{\mu}+1}\int_{B_{(1-\eta)r}(z_0)}(\widetilde{u}^2+\widetilde{v}^2)dz\nonumber
\\&\leq& \left(\frac{1+\eta}{1-\eta}\right)^{C_1(\sqrt{\mu}+1)}C_2(\mu^2+
\eta^{-4}r^{-4})\int_{B_{r}(z_0)}\widetilde{u}^2 dz\nonumber
\\&\leq&(1+3\eta)^{C_1\sqrt{\mu}}C_2\left(\mu^2+\eta^{-4}r^{-4}\right)\int_{B_{r}(z_0)}\widetilde{u}^2dz.
\end{eqnarray}
which is the desired result.
\end{proof}

\begin{remark}\label{remark new doubling condition of u}
From the relationship between $\widetilde{u}$ and $u$, one can obtian that for any $\eta\in\left(0,\frac{1}{3}\right)$,
\begin{equation}\label{new doubling condition of u}
\int_{B_{(1+\eta)r}(x_0)}u^2dx\leq (1+3\eta)^{C\sqrt{\mu}}C\left(\mu^2+\eta^{-4}r^{-4}\right)\int_{B_{r}(x_0)}u^2dx,
\end{equation}
where $B_{(1+\eta)r}(x_0)\subseteq \Omega$, and $C$ is a positive constant depending only on $n$.
\end{remark}

To get the measure estimate of the nodal set of $u$, we also need the following lemma which can be seen in $\cite{Donnelly and Fefferman1}$.
\begin{lemma}\label{basic theorem 4}
Let $f: \ B_1\subseteq \mathbb{C}\rightarrow\mathbb{C}$ be an analytic function with $|f(0)|=1$ and $\sup\limits_{B_1}|f|\leq 2^K$ for some positive constant $K$. Then for any $r\in(0,1)$, the number of zero points of $f$ in $B_r(0)$ is less than or equal to $CK$, where $C$ is a positive constant depending only on $r$.
\end{lemma}

\begin{remark}
In this lemma, it is obvious that the domain $B_1$ is not essential. If one changes $B_1$ into $B_{t}$ for any fixed positive constant $t$, then the conclusion still holds.
\end{remark}

From the new doubling condition in this section, Lemma $\ref{basic theorem 4}$, and the integral geometric formula, which can be found in $\cite{Lin and Yang}$, we can estimate the measure upper bound for the nodal set of $u$ in $\Omega$.

\textbf{Proof of Theorem \ref{nodal set}:}

Let $x_0$ be a point in $\Omega$ and $z_0=(x_0,0)$. Then from Lemma $\ref{upper bound of frequency}$, $N(z_0,R_0)\leq C\sqrt{\mu}$, and $N(z,R_0/2)\leq C\sqrt{\mu}$ for any $z=(x,0)$ with $x\in B_{\frac{R_0}{4}}(x_0)$. Here $R_0$ is a positive constant depending only on $n$ and $\Omega$. Without loss of generality, let $\|\widetilde{u}\|_{L^2(B_{R_0/4}(z_0))}=1$. Then from Lemma $\ref{new doubling condition}$, for any $z\in B_{\frac{R_0}{4}}(z_0)$,

\begin{eqnarray}\label{lower bound of u^2}
\int_{B_{\frac{R_0}{16}}(z)}\widetilde{u}^2dz&\geq&2^{-C(\sqrt{\mu}+1)}\int_{B_{\frac{R_0}{2}}(z)}\widetilde{u}^2dz\\\nonumber
&\geq&2^{-C(\sqrt{\mu}+1)}\int_{B_{\frac{R_0}{4}}(z_0)}\widetilde{u}^2dz\\\nonumber
&=&2^{-C(\sqrt{\mu}+1)}.
\end{eqnarray}
So there exists some point $p_{z}\in B_{\frac{R_0}{16}}(z)$ such that $|\widetilde{u}(p_z)|\geq 2^{-C\sqrt{\mu}}$, since otherwise
\begin{equation}
\int_{B_{\frac{R_0}{16}}(z)}\widetilde{u}^2dz\leq |B_{R_0/16}(z)|2^{-2C\sqrt{\mu}}=CR_0^{n+1}2^{-2C\sqrt{\mu}}.
\end{equation}
This is a contradiction to $(\ref{lower bound of u^2})$, provided that $R_0$ is small enough.
Now choose $z_j\in \partial B_{\frac{R_0}{4}}(z_0)$ on the $x_j$ axis, $j=1,2,\cdots,n+1$. Then for any $j\in\left\{1,2,\cdots,n+1\right\},$ there exists $p_{z_j}\in B_{R_0}(z_j)$ such that $|\widetilde{u}(p_{z_j})|\geq 2^{-C\sqrt{\mu}}$. On the other hand, from the interior estimates, we also have that
$\|\widetilde{u}\|_{L^{\infty}(B_{\frac{R_0}{2}})(z_0)}\leq 2^{C(\sqrt{\mu}+1)}$.

Define $f_j(w;t)=\widetilde{u}(p_{z_j}+tR_0w)$ for $t\in\left(-\frac{5}{16},\frac{5}{16}\right)$ and let $w$ belong to the $n$ dimensional unit sphere. Because each $f_j$ is analytic for $t$, we can extend it to an analytic function $f_j(w;t+i\tau)$ to $|t|<\frac{5}{16}$ and $|\tau|\leq c$, where $c$ is a positive constant depending only on $n$ and $\Omega$. Then from Lemma $\ref{basic theorem 4}$,
\begin{equation*}
\mathcal{H}^0\left\{|t|<\frac{5}{16}\ \big|\ \widetilde{u}(p_{z_j}+tR_0w)=0\right\}\leq C\sqrt{\mu}.
\end{equation*}
Here $\mathcal{H}^0$ is the counting measure. Thus from the integral geometric formula in $\cite{Han and Lin book}$ and $\cite{Lin and Yang}$,
\begin{equation*}
\mathcal{H}^{n}\left(\left\{z\in B_{\frac{R_0}{32}}(z_0)\ \big|\ \widetilde{u}(z)=0\right\}\right)\leq C\sqrt{\mu}R_0^{n},
\end{equation*}
 Because $\widetilde{u}(z)=\widetilde{u}(x,x_{n+1})=u(x)e^{\sqrt{\mu}x_{n+1}}$, and the function $e^{\sqrt{\mu}x_{n+1}}$ is always positive,
\begin{equation*}
\mathcal{H}^{n-1}\left(\left\{z\in B_{\frac{R_0}{64}}(x_0)\ \big|\ u(x)=0\right\}\right)\leq \frac{C}{R_0}\mathcal{H}^{n}\left(\left\{z\in B_{\frac{R_0}{32}}(z_0)\ \big|\ \widetilde{u}(z)=0\right\}\right)\leq C\sqrt{\mu}R_0^{n-1}.
\end{equation*}
Then by covering $\Omega$ with finitely many balls whose radius are $\frac{R_0}{64}$, we have
\begin{equation}\label{inside nodal set}
\mathcal{H}^{n-1}\left(\left\{x\in \Omega\ \big|\ u(x)=0\right\}\right)\leq C\sqrt{\mu}R_0^{-1}\leq C'\sqrt{\mu},
\end{equation}
which is the desired result. 
\qed

\section{Propagation of smallness}

In this section, we will discuss the propagation of smallness of $u$, i.e., we will prove Theorem $\ref{propagation of smallness}$. We do not assume that $\partial\Omega$ is analytic, the frequency function and the doubling index are defined only inside $\Omega$.
We first need the three sphere inequality below.

\begin{lemma}\label{three spheres inequality}
Let $\widetilde{u}$ and $\widetilde{v}$ satisfy $(\ref{new equations with mu})$, $r_0$ be the same positive constant as in Lemma $\ref{monotonicity formula}$. Then for any $r_1<r_2<r_3<r_0$ and $z_0=(x_0,0)$ with $x_0\in \Omega$ and $B_{r_0}(x_0)\subseteq\Omega$, we have
\begin{equation}
\begin{cases}
\|\widetilde{u}\|^2_{L^2(B_{r_2}(z_0))}+\|\widetilde{v}\|^2_{L^2(B_{r_2}(z_0))}\leq Q(\alpha)\left(\|\widetilde{u}\|^{2}_{L^2(B_{r_1}(z_0))}+\|\widetilde{v}\|^2_{L^2(B_{r_1}(z_0))}\right)^{\alpha}\left(\|\widetilde{u}\|^{2}_{L^2(B_{r_3}(z_0))}+\|\widetilde{v}\|^2_{L^2(B_{r_3}(z_0))}\right)^{1-\alpha},\\
\|\widetilde{u}\|_{L^2(B_{r_2}(z_0))}\leq P(\beta)\|\widetilde{u}\|^{\beta}_{L^2(B_{r_1}(z_0))}\|\widetilde{u}\|^{1-\beta}_{L^2(B_{r_3}(z_0))},
\end{cases}
\end{equation} 
where $$Q(\alpha)=\frac{(r_2/r_1)^\alpha}{(r_3/r_2)^{1-\alpha}}\left(\frac{r_2}{r_1}\right)^{\frac{C_2}{\alpha}},$$  $$\alpha=\frac{\ln(r_2/r_1)}{\ln(r_2/r_1)+C_1\ln(r_3/r_2)}\in(0,1),$$ $$P(\beta)=C(\mu+r_1^{-2})^{\beta}(\mu+(r_3-r_2)^{-2})^{1-\beta}\frac{(2r_2/r_1)^{\beta}}{((r_3+r_2)/(2r_2))^{1-\beta}}\left(\frac{2r_2}{r_1}\right)^{\frac{C_2}{\beta}},$$ and $$\beta=\frac{\ln(2r_2/r_1)}{\ln(2r_2/r_1)+C_1\ln((r_3+r_2)/(2r_2))}\in(0,1).$$
Here $C$, $C_1$, and $C_2$ are positive constants depending only on $n$.
\end{lemma}

\begin{proof}
Since $\partial\Omega$ is analytic, the conclusion of Lemma $\ref{monotonicity formula}$ also holds when $z_0=(x_0,0)$ with $B_r(x_0)\subseteq\Omega_R=\{x\ |\ dist(x,\Omega)<R\}$ with $R<r_0$, where $r_0$ is the same positive constant as in Lemma $\ref{monotonicity formula}$. So from Lemma $\ref{monotonicity formula}$ and the definition of the frequency function, we have
\begin{eqnarray}
\ln\frac{H(z_0,r_2)}{H(z_0,r_1)}&=&\int_{r_1}^{r_2}\frac{H'(z_0,r)}{H(z_0,r)}dr=(n-1)\ln\frac{r_2}{r_1}+2\int_{r_1}^{r_2}\frac{N(z_0,r)}{r}dr\nonumber
\\&\leq&(n-1)\ln\frac{r_2}{r_1}+C\left(N(z_0,r_2)+C_0\right)\ln\frac{r_2}{r_1},
\end{eqnarray}
and
\begin{eqnarray}
\ln\frac{H(z_0,r_3)}{H(z_0,r_2)}&=&\int_{r_2}^{r_3}\frac{H'(z_0,r)}{H(z_0,r)}dr=(n-1)\ln\frac{r_3}{r_2}+2\int_{r_2}^{r_3}\frac{N(z_0,r)}{r}dr\nonumber
\\&\geq&(n-1)\ln\frac{r_3}{r_2}+C^{-1}(N(z_0,r_2)-C_0)\ln\frac{r_3}{r_2}.
\end{eqnarray}
Thus we obtain the three sphere inequality of $H(z_0,r)$:
\begin{equation}
H(z_0,r_2)\leq Q'(\alpha)H(z_0,r_1)^{\alpha}H(z_0,r_3)^{1-\alpha}.
\end{equation}
Here $\alpha=\frac{\ln(r_2/r_1)}{\ln(r_2/r_1)+C_1\ln(r_3/r_2)}$, $Q'(\alpha)=\left(\frac{r_2}{r_1}\right)^{\frac{C_2}{\alpha}}$, $C_1$ and $C_2$ are positive constants depending only on $n$. By the integration of $H(z_0,r)$, we have
\begin{equation}
\|\widetilde{u}\|^2_{L^2(B_{r_2}(z_0))}+\|\widetilde{v}\|^2_{L^2(B_{r_2}(z_0))}\leq Q'(\beta)\left(\|\widetilde{u}\|^2_{L^2(B_{r_1}(z_0))}+\|\widetilde{v}\|^2_{L^2(B_{r_1}(z_0))}\right)^{\alpha}\left(\|\widetilde{u}\|^2_{L^2(B_{r_3}(z_0))}+\|\widetilde{v}\|^2_{L^2(B_{r_3}(z_0))}\right)^{1-\alpha},
\end{equation}
where $Q(\alpha)=Q'(\alpha)\frac{(r_2/r_1)^\alpha}{(r_3/r_2)^{1-\alpha}}$. This is the first inequality of this Lemma. The second inequality comes from Lemma $\ref{interior estimate}$ and the first inequality by replacing $r_1$ with $r_1/2$ and $r_3$ with $(r_2+r_3)/2$.
\end{proof}

\begin{remark}
The following three sphere inequality of $u$ can also be obtained by Lemma $\ref{three spheres inequality}$ and the relationship between $u$ and $\widetilde{u}$.
\begin{equation}
\|u\|_{L^2(B_{r_2}(x_0))}\leq S(\theta)e^{C_3\sqrt{\mu}r_0}\|u\|^{\theta}_{L^2(B_{r_1}(x_0))}\|u\|^{1-\theta}_{B_{r_3}(x_0)},
\end{equation}
where 
$$S(\theta)=C(\mu+r_1^{-2})^{\theta}(\mu+(r_3-2r_2+r_1)^{-2})^{1-\theta}\frac{((4r_2-2r_1)/r_1)^{\theta}}{((r_3+2r_2-r_1)/(4r_2-2r_1))^{1-\theta}}\left(\frac{4r_2-2r_1}{r_1}\right)^{\frac{C_2}{\theta}},$$ and $$\theta=\frac{\ln((4r_2-2r_1)/r_1)}{\ln((4r_2-2r_1)/r_1)+C_1\ln((r_3+2r_2-r_1)/(4r_2-2r_1))}.$$ Here $C$, $C_1$ and $C_2$ are positive constants depending only on $n$.
\end{remark}

By the above three sphere inequality, we can prove the propagation of the smallness property of $u$ from some ball $B_{r_0}(x_0)$ to a subset $G\subset\subset\Omega$ as follows.

\begin{lemma}\label{propagation of smallness of ball}
Let $u$ solve $(\ref{basic equations})$, $G$ be a connected open set,  $G\subset\subset\Omega$, and $x_0\subseteq\Omega$.
Assume that
\begin{equation}
\|u\|_{L^{\infty}(B_{r}(x_0))}\leq \eta, \quad\|u\|_{L^{\infty}(\Omega)}\leq 1,
\end{equation}
where $r<dist(G,\partial\Omega)$. Then we have
\begin{equation}
\|u\|_{L^{\infty}(G)}\leq e^{C_1(\sqrt{\mu}r-\ln r)}\eta^{\delta},
\end{equation}
with $\delta=e^{\frac{-C_2diam(\Omega)}{r}}$. Here $C_1$ and $C_2$ are positive constants depending only on $n$.%, and $diam(\Omega)$ means the diameter of $\Omega$.
\end{lemma}

\begin{proof}

For any $h>0$, let $G^{h}$ be the $h$ neighborhood of $G$, i.e., $G^{h}=\left\{x\in \Omega\ \big|\ dist(x,G)<h\right\}$. We also fix $r_3=\frac{r}{2}$, $r_2=\frac{r_3}{2}$ and $r_1=\frac{r_2}{3}$. Now we consider the set $G^{r_1}$. For any $y_0\in G^{r_1}$, there exists a continuous path $\gamma$ from $[0,1]$ to $\Omega$ such that $\gamma(0)=x_0$ and $\gamma(1)=y_0$. Let $0=t_0<t_1<t_2<\cdots<t_K=1$ such that $x_k=\gamma(t_k)$, and $t_{k+1}=\max\{t\ |\ |\gamma(t)-x_k|=2r_1\}$ if $|x_k-y_0|>2r_1$, otherwise we stop the process and set $K=k+1$ and $t_K=1$. Then $\{B_{r_1}(x_k)\}$ are mutually  disjoint balls, $|x_{k+1}-x_k|=2r_1$ for any $k=0,1,2,\cdots,K-1$, and $B_{r_1}(x_{k+1})\subseteq B_{r_2}(x_k)$ for $k=0,1,2,\cdots,K-1$, since $r_1=\frac{r_2}{3}$. From the first inequality of Lemma $\ref{three spheres inequality}$, we have for any $k=0,1,2,\cdots,K-1$,
\begin{eqnarray}
&&\|\widetilde{u}\|^2_{L^{2}(B_{r_1}(x_{k+1}))}+\|\widetilde{v}\|^2_{L^2(B_{r_1}(x_{k+1}))}\leq\|\widetilde{u}\|^2_{L^{2}(B_{r_2}(x_k))}+\|\widetilde{v}\|^2_{L^2(B_{r_2}(x_k))}\nonumber
\\&&\quad\leq  Q\left(\|\widetilde{u}\|^2_{L^{2}(B_{r_1}(x_k))}+\|\widetilde{v}\|^2_{L^2(B_{r_1}(x_k))}\right)^{\alpha}\left(\|\widetilde{u}\|^2_{L^{2}(B_{r_3}(x_k))}+\|\widetilde{v}\|^2_{L^2(B_{r_3}(x_k))}\right)^{1-\alpha}\nonumber
\\&&\quad\leq Q\left(\|\widetilde{u}\|^2_{L^{2}(B_{r_1}(x_k))}+\|\widetilde{v}\|^2_{L^2(B_{r_1}(x_k))}\right)^{\alpha}\left(\|\widetilde{u}\|^2_{L^{2}(\Omega_{r_3}\times(-r_3,r_3))}+\|\widetilde{v}\|^2_{L^2(\Omega_{r_3}\times(-r_3,r_3))}\right)^{1-\alpha}.
\end{eqnarray}
So if we set $$m_l=\frac{\|\widetilde{u}\|^2_{L^{2}(B_{r_1}(x_{l}))}+\|\widetilde{v}\|^2_{L^2(B_{r_1}(x_l))}}{\|\widetilde{u}\|^2_{L^{2}(\Omega_{r_3}\times(-r_3,r_3))}+\|\widetilde{v}\|^2_{L^2(\Omega_{r_3}\times(-r_3,r_3))}},$$ the above inequality becomes
\begin{equation*}
m_{l+1}\leq Qm_l^{\alpha},\quad l=0,1,\cdots,K-1.
\end{equation*}
Thus 
\begin{equation*}
m_K\leq \widetilde{C}m_0^{\delta},
\end{equation*}
where $\widetilde{C}=Q^{c_1}$ with $c_1=\frac{1}{1-\alpha}\geq1+\alpha+\alpha^2+\cdots+\alpha^{K-1}$, and $\delta=\alpha^K$. Hence from Lemma $\ref{interior estimate}$ and the $L^{\infty}$ estimate of $\widetilde{u}$, we obtain that for $z_K=(y_0,0)$,
\begin{equation}
\|\widetilde{u}\|_{L^{\infty}(B_{r_1}(z_K))}\leq e^{C(\ln\mu-\ln r)}\widetilde{C}\left(\|\widetilde{u}\|_{L^{\infty}(B_{2r_1}(z_0))}\right)^{\delta}\left(\|\widetilde{u}\|_{L^{2}(\Omega\times(-r,r))}\right)^{1-\delta}.
\end{equation}
Since $\{B_{r_1}(x_k)\}$ are pairwise disjoint balls and $r_1=\frac{r}{12}$, we have $K\leq \frac{C_1diam(\Omega)}{r}$. Hence $\widetilde{C}= Q^{\frac{1}{1-\alpha}}(\alpha)$ and $\delta=\alpha^{\frac{C_1diam(\Omega)}{r}}$. So from the relationship between $u$ and $\widetilde{u}$, there holds
\begin{eqnarray}
\|u\|_{L^{\infty}(G)}\leq e^{C(\sqrt{\mu}r+\ln\mu-\ln r)}Q^{\frac{1}{1-\alpha}}\|u\|_{L^{\infty}(B_{r}(x_0))}^{\delta}\|u\|_{L^{\infty}(\Omega)}^{1-\delta}
\end{eqnarray}
Here $C$ is a positive constant depending only on $n$ and $\Omega$. This completes the proof.
\end{proof}

From this Lemma, we prove Theorem $\ref{propagation of smallness}$ as follows.

\textbf{Proof of Theorem $\ref{propagation of smallness}$:}

Since $E$ is a convex subset of $\Omega$, there exists a ball $B_r(x_0)$ contained in $E$ with $r<\min\{C\frac{\mathcal{H}^n(E)}{diam(\Omega)^{n-1}},dist(G,\partial\Omega)\}$. Thus the conclusion is obtained by Lemma $\ref{propagation of smallness of ball}$.

\qed

\begin{remark}
By the same arguments as in $\cite{Logunov and Milinicova}$, a similar result of Theorem $\ref{propagation of smallness}$ also holds when we replace the condition on ``$E$ is an open subset of $\Omega$ with $\mathcal{H}^n(E)\geq\epsilon$'' by that ``$E$ is any subset of $\Omega$ with $\mathcal{H}^{n-1+s}(E)>\epsilon$'' for any $s\in(0,1]$. In this case, the positive constants $C$ and $\delta$ in $(\ref{propagation of smallness of G and E})$  depend on $n$, $diam(\Omega)$, $dist(G,\partial\Omega)$, $\epsilon$ and $s$.
\end{remark}

%Since $E$ is a connected convex open subset of $\Omega$ with $dist(E,\partial\Omega)>\rho$, there exists a ball $B_{r_0}(x_0)\subseteq E$ with $r_0\geq \frac{\mathcal{H}^n(E)}{dist(\Omega)^{n-1}}$, where $C$ is a positive constant depending only on $n$. Let $G'$ be a subset of $\Omega$ such that $E\cup G\subseteq G'$ with $dist(G',\partial\Omega)>\frac{\rho}{2}$. Let $\bar{\rho}=\min\{\rho,r_0\}$. Then from Lemma $\ref{propagation of smallness of ball}$, we have
%\begin{equation*}
%\|u\|_{L^{\infty}(G)}\leq \|u\|_{L^{\infty}(G')}\leq e^{C(\sqrt{\mu}-\ln\bar{\rho}+1)}\|u\|_{L^{\infty}(B_{r_0}(x_0))}\leq e^{C(\sqrt{\mu}-\ln\bar{\rho}+1)}\|u\|^{\delta}_{L^{\infty}(E)},
%\end{equation*}
%where $\delta=e^{-C\frac{diam(\Omega)}{\bar{\rho}}}$. This completes the proof.

\section*{Acknowledgement}

This work is supported by the National Natural Science Foundation of China (Grant No. 12071219 and No. 11971229).
%再增加一些基金号，比如张院长的项目，比如梁老师的项目。

%%%%%%%%%%%%%%%%%%%%%%%%%%%%%%%%%%%%%%%%%%%%%%%%%%%%%%%%%%

\end{document}